\documentclass[onefignum,onetabnum]{siamart171218}

\usepackage{lipsum}
\usepackage{amsfonts}
\usepackage{graphicx}
\usepackage{epstopdf}
\ifpdf
  \DeclareGraphicsExtensions{.eps,.pdf,.png,.jpg}
\else
  \DeclareGraphicsExtensions{.eps}
\fi

\usepackage[nice]{nicefrac}

\newsiamremark{remark}{Remark}
\newsiamremark{hypothesis}{Hypothesis}
\crefname{hypothesis}{Hypothesis}{Hypotheses}
\newsiamthm{claim}{Claim}

\headers{Error analysis of BDF methods for transient Stokes}{A. Contri, B. Kov\'{a}cs, and A. Massing}

\title{Error analysis of BDF 1--6 time-stepping methods for the transient Stokes problem:
\\
Velocity and pressure estimates\thanks{Submitted to the editors \today.
\funding{The work of B.K.~is funded by DFG Heisenberg Programme (Pr.-ID 446431602).}}}

\author{Alessandro Contri\thanks{Department of Mathematical Sciences, NTNU, Trondheim, Norway 
  (\email{alessandro.contri@ntnu.no}).}
\and Bal\'{a}zs Kov\'{a}cs\thanks{Institut für Mathematik, Universit\"{a}t Paderborn, Paderborn, Germany
(\email{balazs.kovacs@math.uni-paderborn.de}).}
\and Andr\'{e} Massing\thanks{Department of Mathematical Sciences, NTNU, Trondheim, Norway 
  (\email{andre.massing@ntnu.no}).}}

\makeatletter
\newcommand*{\addFileDependency}[1]{
  \typeout{(#1)}
  \@addtofilelist{#1}
  \IfFileExists{#1}{}{\typeout{No file #1.}}
}
\makeatother

\ifpdf
\hypersetup{
  pdftitle={Error analysis of BDF time stepping methods for the transient Stokes problem},
  pdfauthor={A. Contri, B. Kov\'{a}cs, and A. Massing}
}
\fi

\usepackage{mathsemantics}
\usepackage{verbatim}
\usepackage{subcaption}
\newcommand{\bard}{{\bar{\partial}_q}}
\newcommand\myeq{\mathrel{\stackrel{\makebox[0pt]{\mbox{\normalfont\tiny def}}}{=}}}

\newcommand{\MT}{multiplier technique }
\DeclareMathOperator{\divergence}{div}
\renewcommand{\div}{\divergence}

\usepackage{ulem}

\begin{document}

\maketitle

\begin{abstract}
    We present a new stability and error analysis
    of fully discrete approximation schemes for 
    the transient Stokes equation. 
    For the spatial discretization, we consider a wide
    class of Galerkin finite element methods 
    which includes both inf-sup stable spaces and symmetric pressure stabilized formulations. 
    We extend the results from Burman and Fernández [\textit{SIAM J.\,Numer.\,Anal.}, 47 (2009), pp.~409-439] 
    and provide a unified theoretical analysis of
    backward difference formulae (BDF methods) of order 1 to 6.
    The main novelty of our approach lies in deriving optimal-order stability and 
    error estimates for both the velocity and the pressure
    using Dahlquist's $G$-stability concept
    together with the multiplier technique introduced by Nevanlinna and Odeh 
    and recently by Akrivis et~al.\ [\textit{SIAM J.\,Numer.\;Anal.}, 59 (2021), pp.~2449-2472].
    When combined with a  method dependent Ritz projection for the initial data, unconditional stability can be shown while
    for arbitrary interpolation, pressure stability is subordinate to the fulfilment of a mild inverse CFL-type condition between space and time discretizations. 
\end{abstract}

  

\section{Introduction}
 Let $\Omega\subset\bbR^d,\; d\in\{2,3\}$ be a Lipschitz domain with boundary~$\partial\Omega$ and $J := (0,T),\; T>0$ be a finite time interval.
 We consider the transient Stokes equation as a simplified prototype
 problem for  the time-dependent flow of an incompressible fluid:
 find $\bu \colon \Omega\times J \to\bbR^d$ and
 $p\colon\Omega\times J \to \bbR$ such  that
\begin{equation}
    \left\{
    \begin{alignedat}{4}
        ~ \partial_t \bu-\nu\Delta \bu+\nabla p = &\ \bf && \qquad\text{ in }  \Omega\times J,\\
        \nabla\cdot \bu =&\ 0 &&\qquad\text{ in } \Omega\times J, \\
        \bu=&\ 0  && \qquad\text{ on }  \partial\Omega \times J, \\
        \bu(\cdot,0)=&\ \bu_0 &&\qquad\text{ in } \Omega,
    \end{alignedat}\right.
    \label{eq:1}
\end{equation}
where $\bf\colon\Omega \times J \to\bbR$ is an external force field,
$\bu_0\colon\Omega\to \bbR^d$ is the given initial, divergence-free velocity field and $\nu>0$ denotes a constant kinematic viscosity.

\subsection{State of the art}
This paper is concerned with the stability and error analysis
of fully discrete formulations of the transient Stokes
equation~\eqref{eq:1}.
We analyze BDF methods of order $1$--$6$ in combination
with a large class of Galerkin finite element methods
which utilize so-called symmetric pressure stabilizations.
For the classical analysis of inf-sup stable Galerkin finite elements paired with lower order time stepping schemes,
we refer to \cite[Part XIV]{Ern2021a} and \cite[Chapter 3--7]{John2016}. 
Due to the complexity of constructing mixed finite elements which provably satisfy a discrete uniform inf-sup condition, 
stabilized methods have been developed and analyzed which
circumvent the inf-sup condition and allow for equal-order interpolation spaces.
Here, a main challenge in the design and theoretical analysis of the
fully discrete transient Stokes problem is the intricate interplay
between the selected stabilized space discretization of the Stokes
operator and the chosen time stepping scheme.
For instance, the popular strongly consistent Pressure Stabilized Petrov--Galerkin (PSPG) method, 
proposed in \cite{HughesFrancaBalestra1986}, introduces a coupling 
between the time derivative of the velocity and the pressure gradient in the transient case.
As a result, pressure instabilities can occur when an 
inverse parabolic Courant--Friedrichs--Lewy (CFL) condition between
the time step $\tau$ and the mesh size $h$ in the form $\tau>C
h^{2}/\nu$ is not satisfied. Moreover, $H^1$-conform discrete velocity
spaces are also subject to an inverse parabolic or hyperbolic CFL
condition for BDF-1 or BDF-2 schemes, respectively, see \cite{Burman2011}. 
These small time step instabilities can be avoided by
a careful construction of the initial data~\cite{John2015}.

As an alternative to residual-based stabilization approaches such as the PSPG method,
a number of weakly consistent symmetric pressure stabilizations have been proposed.
They have the favorable property that they typically
commute with the discretization of both temporal operators 
and optimal control problems.
Instances include the pressure gradient stabilization by Brezzi and Pitk\"aranta~\cite{Brezzi1984},
Orthogonal Subscale Stabilizations (OSS) \cite{Codina1997},
methods based on certain pressure projections such as the
Local Projection Stabilization (LPS) \cite{Becker2001}
or the method by Dohrmann and Bochev~\cite{Dohrmann2004},
and methods based on a Continuous Interior Penalty (CIP) for
$H^1$-conform discrete pressure variable~\cite{BurmanHansbo2006a}.
In~\cite{Burman2009},
all the aforementioned symmetric pressure stabilized formulations
were put into a general abstract framework 
to provide a unified analysis of fully discrete
formulations employing either BDF-1, BDF-2 or Crank-Nicolson
time stepping methods.
Optimal convergence rates for the velocity and pressure are derived if
a suitable Ritz projection is used for the initial data, otherwise a
weak inverse CFL condition of the form $\tau>Ch^{2k}$ must be
satisfied, where $k$ is the polynomial degree of the employed velocity
approximation space. The analysis was later
expanded to the full Navier--Stokes case in \cite{GarciaArchilla2020}.

The majority of analyses of either inf-sup stable or stabilized
formulations of the transient Stokes equation are restricted to at most second order in time
schemes such as BDF-1, BDF-2 or Crank-Nicolson methods.
Notable exceptions are discontinuous Galerkin time stepping methods,
and we refer to \cite{ChrysafinosWalkington2010} and
\cite{Ahmed2017} 
for the case of inf-sup stable and LPS stabilized spatial discretizations, respectively,
and to \cite{BehringerVexlerLeykekhman} for an analysis which involves inf-sup stable spaces
and only minimal regularity assumptions on the data.
To the best of our knowledge, the only available
analysis of higher order BDF methods for the transient Stokes problem
is provided
in \cite{Liu2013},
which considers \mbox{BDF-$q$} schemes up to
order 5 on both stationary and moving domains.
The analysis is quite elaborate as it requires the solution of a
nonlinear system of coefficients that are needed to guarantee that
certain telescopic properties of the discrete time derivative and a
positivity condition for the energy bilinear form are upheld when 
suitable discrete velocity test
functions are inserted into
the momentum equation.
The coefficients do not come in close form and
are computed numerically for $q=3,4,5$ while 
the case $q = 6$ could not be treated.
Most importantly, the described approach
does not provide stability or error estimates
for the pressure approximation.
However,
in recent years, 
Dahlquist's $G$-stability theory \cite{Dahlquist1978}
in connection with the \MT devised by
Nevanlinna and Odeh~\cite{Nevanlinna1981}
has developed into a powerful technique to
analyse higher order BDF time stepping methods for parabolic
problems in an elegant and unified way.
Originally developed for the 
stability analysis of linear multistep methods for contractive nonlinear ordinary differential equations,
the use of these techniques for parabolic PDEs using BDF up to
order 5 has been first presented in \cite{Lubich2013}. 
The extension to quasi-linear parabolic PDEs is due to
\cite{Akrivis2015a}. 
These results have been
successively extended to a larger class of problems (see
\cite{Akrivis2017,Akrivis2017a,Akrivis2021}) and the original
estimates by Nevanlinna (\cite{Nevanlinna1978,Nevanlinna1979}) were
sharpened in \cite{Akrivis2015}. The case of BDF-6 is treated in
\cite{Akrivis2021a}. We refer to \cite{Hairer1993,Hairer1996} for the
theory of BDF methods and $G$-stability. For PDEs, fully discrete error
analysis of high-order BDF methods utilizing approaches similar to
what is done in this text are considered in \cite{Lubich2013} with
application to moving surfaces, while~\cite{MCF,Willmore} and \cite{Akrivis2021}
used the multiplier techniques but testing with the time derivatives.
The \MT is also used in the analysis of the so called Scalar
Auxiliary Variable (SAV) method to prove convergence of higher order
time stepping schemes, see for example \cite{Huang2021}.

\subsection{New contributions and outline of the paper}
To the best of our know\-ledge there is no fully discrete analysis of
the transient Stokes problem for BDF-q from 1 to 6 providing
stability and optimal error estimates for both velocity and pressure.
The goal of our paper is therefore to close this gap, and at the same time, we wish to
exemplify how the concept of $G$-stability and multiplier techniques can be
used to analyze time-dependent saddle point problems 
by mimicking arguments from the stability analysis of the continuous problems.
We start therefore by reviewing the weak formulation of the transient Stokes problem
in Section~\ref{subsec:spaces+weakform}, with a particular emphasis
on how bounds for the velocity and its time-derivate
in the respectively $L^2$--$H^1$ and $L^2$--$L^2$ norms
lead to a pressure estimate in the $L^2$--$L^2$ norm via the inf-sup condition.
Section~\ref{sec:preliminaries+notation} 
then provides details on the fully discrete formulation of the transient Stokes problem
and collects the relevant theoretical results for its analysis.
We recall the spatial discretization
framework from~\cite{Burman2009} which includes
a large family of Galerkin finite element methods
including symmetric pressure stabilized and inf-sup stable formulations.
Each spatial discretization yields 
a method-dependent Ritz projection that will play an
important role later in the analysis. 
For the discretization in time, we use 
BDF-$q$ methods of order $q=1,\dotsc,6$, and 
we recall the main results from Dahlquist's theory of $G$-stability \cite{Dahlquist1978}
and multiplier techniques by Nevanlinna--Odeh \cite{Nevanlinna1981} and Akrivis~et~al.~\cite{Akrivis2021a}
that will be crucial for the subsequent stability analysis.  

In Section~\ref{sec:stability},
we start the development of the new stability analysis of the fully discrete transient
Stokes problem.
As a first step,  $L^{\infty}$ in time
bounds for the velocity and pressure with respect to
$H^1$-norm for the velocity
and the stabilization induced semi-norm for the pressure
are established by combining
$G$-stability with the Nevanlinna--Odeh multiplier technique,
similar to analysis in \cite{Lubich2013,Akrivis2015a}.
But in contrast to the parabolic case, we also exploit the
stabilized discrete formulation of the incompressibility 
constraint to recast certain pressure-velocity coupling terms
as pressure contribution measured in a semi-norm.
Next, we establish a bound 
for the discrete time-derivative of the velocity 
with respect to the $L^{2}$--$L^2$ norm.
This is the trickiest part of our stability analysis
as we need to test the weak momentum equation at the current
and previous time-step with suitably scaled discrete time-derivatives
to use the \MT and $G$-stability,
similar to analysis in \cite{MCF,Akrivis2021}. 
Again, the weak incompressibility equation allows us to handle
velocity-pressure coupled terms which arise during the testing procedure.
Moreover, unconditional stability for the pressure is only ensured
if the initial values are divergence-free in a suitable discrete sense.
The final bound for the pressure is derived
from the modified discrete inf-sup condition similar to the continuous case.
Section~\ref{sec:stab-bdf6} is then devoted to extend our stability analysis to BDF-6. 
Since there is no Nevanlinna--Odeh multiplier for BDF-6,
we combine our analysis with the approach from \cite{Akrivis2021a} and use a relaxed
multiplier which in turn  requires a Toepliz matrix based argument to ensure
that certain cross term contributions remain positive after telescoping
as explained in \cite{Akrivis2021a}. However, the final stability estimates
are completely analogous to the ones derived in Section~\ref{sec:stability}.  

The derivation of optimal-order error bounds for both
the velocity and pressure is performed in Section~\ref{sec:convergence}.
The Ritz projection is used to split the total error and allows us 
to rely on the stability estimates of Section~\ref{sec:stability} and Section~\ref{sec:stab-bdf6}. 
Concerning the velocity errors, optimal-order error bounds are derived using the finite
element interpolation of a sufficiently regular initial data. For the
pressure, the stability constraint of discretely divergence-free initial data
results in being crucial in order to obtain optimal-order error
bounds. Afterwards in Section~\ref{sec:small_timestep}, we briefly describe
how the stability analysis can be adapted
if the initial data is given without such a property,
leading to the same small time-step limit phenomena previously analyzed in \cite{Burman2009} 
for low-order time stepping schemes. Finally,
we present a series of numerical
experiments in Section~\ref{section:numerics} which illustrate and complement our theoretical results.

\section{The transient Stokes problem}
\label{subsec:spaces+weakform}
In this section we present the weak formulation of the transient Stokes problem and introduce the notation used throughout this paper. While we focus on the specific case of the transient Stokes problem, our derivations are presented in such a form that they can easily be applied to general time-dependent saddle point problems which exhibit the same structure in their weak formulations. The main purpose of this section is to review the main arguments in the stability analysis for the velocity and the pressure
in the fully continuous setting, as our main contribution is to show how to translate these arguments 
to the full discretization to prove stability and convergence of the numerical method.

\subsection{Notation}
\label{subsec:notation}
Throughout the paper the standard notation for Sobolev spaces  is used (see, for example, \cite{Ern2021}). Given a measurable domain $\Omega\in\mathbb{R}^d,\; d=2,3$ the norm (resp.~semi-norm) in $W^{m,p}(\Omega)$ is denoted by $\norm[auto]{\;\cdot\;}_{m,p}$
(resp.~$|{\;\cdot\;}|_{m,p}$). The case $p=2$ will be distinguished by using $H^m(\Omega)$ to denote the space $W^{m,2}(\Omega)$  and by using the norm (resp.~semi-norm) notation $\norm[auto]{\;\cdot\;}_{m}$ (resp.~$|{\;\cdot\;}|_{m}$). The space $H_0^1(\Omega)$ is the closure in $H^1(\Omega)$ of the set of infinitely
differentiable functions with compact support in $\Omega$, and $H^{-1}(\Omega)$ is its dual. We denote by $\inner[auto]{\cdot}{\cdot}$ the inner product in $[L^2(\Omega)]^d$ (the domain of integration $\Omega$ is taken 
for granted) and we define $L^2_0(\Omega)=\{q\in L^2(\Omega): \inner{q}{1}=0\}$.
Given the Banach space $W$ and a time interval $J$, we denote the corresponding Bochner space, i.e.~the space of
functions defined on the time interval $J$ with values in $W$, by $L^p(J;W)$ with norm
$\norm{\;\cdot\;}_{L^p(J,W)}$. We also define $H^q(J,W):=\{\partial_t^{(i)}v\in L^2(J;W), \; i=0,\ldots,q\}$.
Identifying $[L^2(\Omega)]^d$ with its dual, the spaces $H_0^1(\Omega) \subset L^2(\Omega) \subset H^{-1}(\Omega)$ constitutes a Gelfand triple.

\subsection{Weak formulation}
\label{subsec:weak-form-trans-stokes}
We recast the transient Stokes problem into a general
time-dependent saddle point problem in weak formulation where we seek
$(\bu(t), \; p(t))\in V \times Q$ such that
\begin{subequations}
	\label{eq:2}
    \begin{align}
        \inner[auto]{\partial_t\bu(t)}{\bv}+a\inner[auto]{\bu(t)}{\bv}+b\inner[auto]{p(t)}{\bv}&=
        \inner[auto]{\bf(t)}{\bv} &&\text{a.e. in }J, \label{eq:2a}\\
        b\inner[auto]{q}{\bu(t)}&=0 && \text{a.e. in } J, \label{eq:2b}\\
        \bu(0)&=\bu_0 && \text{a.e. in }\Omega,
    \end{align}
\end{subequations}
for all $(\bv,\;q)\in V\times Q$ and $t>0$. Here, $V$ and $Q$ are
appropriate Hilbert spaces endowed with the respective norms
$\norm{\cdot}_V$ and $\norm{\cdot}_Q$. 
As usual, we assume that the involved bilinear forms
are bounded so that
\begin{equation}
    a\inner[auto]{\bu}{\bv}\leqslant C_1
    \norm[auto]{\bu}_V\norm[auto]{\bv}_V, 
    \qquad 
    b\inner[auto]{\bv}{q}\leqslant C_2\norm[auto]{\bv}_V\norm[auto]{q}_Q,
    \label{eq:3}
\end{equation}
holds for some positive finite constants $C_1,C_2 \geqslant 0$. 
To ensure well-posedness of \eqref{eq:2}, we require that
$a\inner{\cdot}{\cdot}$ is coercive and that the
Ladyzhenskaya--Bab\^{u}ska--Brezzi (LBB) condition (also known as
inf-sup condition) holds for $b\inner[auto]{\cdot}{\cdot}$,
\begin{subequations}
	\label{eq:4}
    \begin{alignat}{2}
    &a\inner[auto]{\bv}{\bv}\geqslant \alpha\norm{\bv}_V^2, && \quad \forall \bv\in V,
    \label{eq:4a}
    \\
    &\beta\norm[auto]{q}_Q\leqslant \sup_{\bv\in V}\frac{\abs{b\inner[auto]{q}{\bv}}}{\norm[auto]{\bv}_V} && \quad \forall q\in Q,
    \label{eq:4b}
    \end{alignat}
\end{subequations}
for some constants $\alpha, \beta > 0$.
To ease the notation, most inequalities will be written
down with generic constants $C$ that are independent of the mesh
size and data, and we will be explicit only when the constants 
are crucial for the analysis.
For the concrete case of the transient Stokes equation, the bilinear forms in \eqref{eq:2} are given by
\begin{equation}
    a\inner[auto]{\bu}{\bv}=\inner[auto]{\nu\nabla \bu}{\nabla \bv}, \qquad 
    b\inner[auto]{p}{\bv}=-\inner[auto]{p}{\nabla\cdot \bv},
    \label{eq:5}
\end{equation}
where following time independent function spaces and associated norms are employed,
\begin{gather}
    H=[L^2(\Omega)]^d, \quad V=[H_0^1(\Omega)]^d, \quad Q=L_0^2(\Omega),
    \label{eq:6}
    \\
    \norm[auto]{\bv}_H=\inner[auto]{\bv}{\bv}^{\frac{1}{2}},\quad 
    \norm[auto]{\bv}_V=\norm{\nu^{\frac{1}{2}}\nabla \bv}_H=a\inner{\bv}{\bv}^{\frac{1}{2}}, \quad 
    \norm[auto]{q}_Q=\norm{\nu^{-\frac{1}{2}}q}_H.
    \label{eq:7}
\end{gather}
In particular note that the norm $\norm{\cdot}_V$ identifies with the so called energy norm resulting in the equality $a\inner[auto]{\bv}{\bv}=\norm{\bv}_V^2$. By imposing homogeneous Dirichlet boundary conditions on the whole boundary, the Poincaré inequality $\norm{\bv}_H\leqslant c_P \nu^{-\frac{1}{2}}\norm{\bv}_V$ holds, see, e.g.,~\cite[Lemma 3.27]{Ern2021}.
\begin{remark}
    The analysis that will follow depends on the fact that we work with the equality $a\inner[auto]{\bv}{\bv}=\norm{\bv}_V^2$. 
    In the case of a general saddle point problem posed on Hilbert spaces where assumptions \eqref{eq:3} and \eqref{eq:4} hold, we can always work with the energy norm defined by $a\inner{\bv}{\bv}=\norm{\cdot}_V^2$ and continue to have a well-posed problem. More details can be found in \cite[Section~49.2.3]{Ern2021b}.
\end{remark}

\subsection{Stability of the transient Stokes problem}
\label{subsection:continuous stability}
As usual, the a priori error estimates 
in this paper will be shown by separating the issues of stability and consistency. The stability proofs then all aim to bound the numerical solution in terms of forcing terms and initial values. 
Here,  we review the basic mechanism behind the stability proofs
for the fully continuous Stokes problem,
which we subsequently will try to mimic in order
to derive a priori stability estimates for the numerical method. 
This allows us to sketch the big picture used in the fully discrete stability analysis proposed in this article.
It is our intention to derive a priori stability estimates for both
the velocity and the pressure, where
the latter requires the stability of
the velocity derivative. For the continuous case it is possible to
derive an a priori bound in $L^2(J;V')$ by using time
distributional derivatives, see \cite[Section~72.4]{Ern2021a}.
For simplicity we follow the analysis in \cite[Section~72.3]{Ern2021a}
and choose $\bf\in L^2(J;H)$ instead.
 We also define $V_{0,\div}:=\{
\bv\in V | \nabla\cdot\bv =0 \}$ and require $\bu_0\in V_{0,\div}$
together with $(\bu,p)$ in $L^2(J;V)\times L^2(J;Q)$. An equivalent
reformulation of \eqref{eq:2} is then that
\begin{equation}
    \int_J\bigl(\inner[auto]{\partial_t\bu(t)}{\bv}+a\inner[auto]{\bu(t)}{\bv}+b\inner[auto]{p(t)}{\bv} -  b\inner[auto]{q}{\bu(t)} \bigr)\d t = \int_J \inner[auto]{\bf(t)}{\bv} \d t
    \label{eq:8}
\end{equation} 
holds for all $\bv \in L^2(J;V)$ and
all $q\in L^2(J;Q)$.
Under all these assumptions \eqref{eq:2} is well posed and admits a unique solution, see \cite[Section~72.2]{Ern2021a}. We also have that $\partial_t \bu\in L^2(J;H)$ and that \eqref{eq:2b} implies $\bu \in L^2(J,V_{0,\div})$.

\subsubsection*{Velocity bound}  Assuming sufficient regularity of the solutions,
we start by choosing as test functions $\bv = \bu(t) \in V$ and $q = p(t) \in Q$
in the weak form~\eqref{eq:2}.
Subtracting then~\eqref{eq:2b} from \eqref{eq:2a}, the contributions of the bilinear form $b\inner[auto]{\cdot}{\cdot}$ cancel out, yielding
\begin{align}
    \dfrac{1}{2}\dfrac{\d}{\d t}\norm[auto]{\bu(t)}^2+\norm[auto]{\bu(t)}_V^2 
    &= \inner[auto]{\partial_t \bu(t)}{\bu(t)} + a\inner[auto]{\bu(t)}{\bu(t)} = \inner[auto]{\bf(t)}{\bu(t)}
    \\
    &\leqslant \frac{c_P^2}{2\nu}\norm[auto]{\bf(t)}^2_{H}+\frac{1}{2}\norm[auto]{\bu(t)}^2_{V},
    \label{eq:9}
\end{align}
after the use of a Cauchy--Schwarz inequality, the Poincaré inequality and modified Young's inequalities. A successive integration in time over $J=(0,T)$ and rearrangement of terms gives the bound 
\begin{equation}
    \norm[auto]{\bu(T)}^2_H+\norm[auto]{\bu}_{L^2(J;V)}^2 \leqslant \norm[auto]{\bu_0}^2_H+\frac{c_P^2}{\nu}\norm[auto]{\bf}_{L^2(J;H)}^2.
    \label{eq:10}
\end{equation}

\subsubsection*{Pressure bound} A bound for the pressure can be obtained
starting from the inf-sup condition~\eqref{eq:4},
\begin{equation}
    \label{eq:11}
    \begin{aligned}
    \norm[auto]{p(t)}_Q&\leqslant \frac{1}{\beta} \sup_{\bv\in V}\frac{b(p(t),\bv)}{\norm[auto]{\bv}_V} =  \frac{1}{\beta} \sup_{\bv\in V}\frac{\inner[auto]{\bf(t)}{\bv}-a\inner[auto]{\bu(t)}{\bv}-\inner[auto]{\partial_t\bu(t)}{\bv}}{\norm[auto]{\bv}_V}  \\
    &\leqslant \frac{1}{\beta} \sup_{\bv\in V}\frac{\norm[auto]{\bf(t)}_{H}\norm[auto]{\bv}_H+\norm[auto]{\bu(t)}_V\norm[auto]{\bv}_V+\norm[auto]{\partial_t\bu(t)}_H\norm[auto]{\bv}_H}{\norm[auto]{\bv}_V}  \\
    &= \frac{1}{\beta}\left[ \frac{c_P}{\nu^{\frac{1}{2}}}\norm[auto]{\bf(t)}_{H}+ \norm[auto]{\bu(t)}_V + \frac{c_P}{\nu^{\frac{1}{2}}} \norm[auto]{\partial_t\bu(t)}_H \right].
    \end{aligned}
\end{equation}
We recognize at this point that a bound for $\partial_t\bu(t)$ 
in the $H$-norm is needed, which can be obtained
by formally testing \eqref{eq:2} with
$v = \partial_t\bu(t)$ and $q = 0$ leading to
\begin{align}
     \norm[auto]{\partial_t\bu(t)}_H^2+\frac{1}{2}\frac{\d}{\d t}\norm[auto]{\bu(t)}_V^2
    &=\norm[auto]{\partial_t\bu(t)}_H^2+a\inner[auto]{\bu(t)}{\partial_t\bu(t)}
    =\inner{\bf(t)}{\partial_t\bu(t)} 
    \\
    &\leqslant \frac{1}{2}\norm[auto]{\bf(t)}_H^2  + \frac{1}{2}\norm[auto]{\partial_t \bu(t)}_H^2
    \label{eq:12}
\end{align}
Given the initial condition $\bu_{0,\div}\in V_{0,\div}$ and integrating in $J$ leads to
\begin{equation}
    \norm[auto]{\partial_t\bu}_{L^2(J;H)}^2+\norm{\bu(T)}_V^2\leqslant \norm{\bu_0}_V^2 + \norm{\bf}_{L^2(J;H)}^2.
    \label{eq:14}
\end{equation}
Technically, $\partial_t\bu(t) \notin V$ and cannot be inserted
directly into the bilinear forms $a\inner{\cdot}{\cdot}$,
$b\inner{\cdot}{\cdot}$. We thus have to follow a semi-discrete
Galerkin-type argument and appeal to compactness results in order to
arrive to \eqref{eq:14} and we refer to
\cite[p.234-239]{Ern2021a} for the details.
As only a bound on the velocity derivative in terms of the Bochner norm is obtained, 
the pressure is typically estimated in the same way by integrating \eqref{eq:11} over $J$,
\begin{equation}
    \begin{aligned}
    \norm[auto]{p}_{L^2(J;Q)}^2 &\leqslant \frac{1}{\beta^2} \int_0^t \left[ \frac{c_P}{\nu^{\frac{1}{2}}}\norm[auto]{\bf(s)}_{H}+ \norm[auto]{\bu(s)}_V + \frac{c_P}{\nu^{\frac{1}{2}}} \norm[auto]{\partial_t\bu(s)}_H \right]^2 \d s \\
    &\leqslant \frac{3}{\beta^2} \int_0^t \left[ \frac{c_P^2}{\nu}\norm[auto]{\bf(s)}_{H}^2+ \norm[auto]{\bu(s)}_V^2 + \frac{c_P^2}{\nu} \norm[auto]{\partial_t\bu(s)}_H^2 \right] \d s \\
    & \leqslant \frac{3}{\beta^2} \left[\norm[auto]{\bu_0}_H^2 + \frac{c_P^2}{\nu}\norm[auto]{\bu_0}_V^2 + 3\frac{c_P^2}{\nu} \norm[auto]{\bf}_{L^2(J;H)}^2 \right].
    \end{aligned}
    \label{eq:16}
\end{equation}

\subsubsection*{Combined stability bound}
Summing \eqref{eq:10} and \eqref{eq:16}, we obtain the following stability estimate
for sufficiently regular solutions of the transient Stokes equation
\begin{equation}
\label{eq:stability estimate - continuous}
	\begin{aligned}
		  \norm[auto]{\bu(T)}^2_H+\norm[auto]{\bu}_{L^2(J;V)}^2
		 + \norm[auto]{p}_{L^2(J;Q)}^2
		  \leqslant C \left[ \norm[auto]{\bu_0}_H^2 + \frac{c_P^2}{\nu}\norm[auto]{\bu_0}_V^2 + \frac{c_P^2}{\nu}\norm[auto]{\bf}_{L^2(J;H)}^2\right] . 
	\end{aligned}
\end{equation} 

\section{Preliminaries and notation}
\label{sec:preliminaries+notation}
In this section, we collect the main assumptions 
and theoretical tools for the spatial and temporal discretization
approaches considered in this paper.
First, we briefly recap the spatial discretization framework of
Galerkin finite element methods with symmetric pressure stabilization.
Afterwards, the $q$-step Backward Difference Formulas and related stability 
concepts are reviewed.

\subsection{Space semi-discretization based on symmetric pressure stabilization}
We proceed with describing the spatial semi-discretization of the problem where we closely follow the presentation in~\cite{Burman2009}. 
Let $\{\cT_h\}_{0<h\leqslant 1}$ be a family of quasi-uniform triangulations of the 
domain $\Omega$. The subscript $h$ refers to the mesh size $h=\max_{T\in\mathcal{T}_h}h_T$ 
of $\cT_h$ where $h_T$ is
the diameter of a mesh cell $T \in \cT_h$.
We define the spaces of continuous and (possibly) discontinuous piecewise
polynomial functions of degree $k\geqslant 1$ and $l\geqslant 0\; (k-1\leqslant l\leqslant k)$, respectively,
\begin{subequations}
	\label{eq:17}
	\begin{align}
	    X_h&=\{\bv_h\in C(\Omega)\;:\;\bv_{h}\in\bbP_c^k(\mathcal{T}_h) \text{ and } \bv_{h}|_T\in\bbP^k(T)\; \forall T\in \cT_h\}, \label{eq:17a}
	    \\
	    M_h&=\{q_h\in L^2(\Omega)\;:\;q_{h}\in\bbP_{dc}^l(\cT_h)\}. \label{eq:17b}
	\end{align}
\end{subequations}
The approximate velocities will belong to the space $V_h=[X_h]^d\cap V$,
and for the pressure we will use either $Q_h=M_h\cap Q$ or $Q_h=M_h
\cap Q \cap C^0(\overline{\Omega})$.

To handle mixed finite element spaces which violate the discrete counterpart of the inf-sup condition~\eqref{eq:4b},
we allow for the addition of a pressure stabilization form
$j_h\colon Q_h\times Q_h\to \bbR$ to the discrete weak formulation of the transient Stokes problem.
We require that the stabilization form $j_h$ is
symmetric, positive semi-definite and
satisfy the boundedness properties
\begin{align}
    \abs{j_h\inner[auto]{p_h}{q_h}}&\leqslant \abs{p_h}_{j_h}\abs{q_h}_{j_h}\leqslant C\norm[auto]{p_h}_Q\norm[auto]{q_h}_Q \quad &\forall p_h,q_h\in Q_h, \label{eq:18b} 
\end{align}
where the first inequality for the $j_h$-induced semi-norms $\abs{\cdot}_{j_h}$
is automatically satisfied.
Next, we assume the existence of a quasi-interpolation operator 
$\cI_h^p \colon Q\to Q_h$
such that
\begin{alignat}{3}
    \abs[auto]{\cI_h^p q}_{j_h}^{\frac{1}{2}}
    &\leqslant C\frac{h^{s_p}}{\nu}\norm[auto]{q}_{s_p} && \quad \forall q\in H^{s}(\Omega)
    \label{eq:18c}
    \\
    \norm[auto]{q-\cI_h^p q}_{Q}&\leqslant C \frac{h^{l+1}}{\nu^{\frac{1}{2}}}\norm[auto]{q}_{l+1}
    && \quad \forall q\in H^{l+1}(\Omega).
    \label{eq:19}
\end{alignat}
where $s_p\myeq\min\{s,\bar{l},l+1\},\; \bar{l}\geqslant 1$ is the order of weak consistency
of the stabilization operator.
We also assume the existence of a quasi-interpolation operator $\cI_h^\bu \colon V\to V_h$ 
satisfying the approximation properties
\begin{subequations}
	\label{eq:20}
    \begin{align}
    &\norm[auto]{\bv-\cI_h^\bu \bv}_H+h\nu^{-\frac{1}{2}}\norm[auto]{\bv-\cI_h^\bu \bv}_V\leqslant C_\cI h^{r_\bu} \norm[auto]{\bv}_{r_\bu} , \label{eq:20a}\\
    &\abs{b\inner[auto]{q_h}{\bv-\cI_h^\bu \bv}}\leqslant C\abs[auto]{q_h}_{j_h}\left(
        \nu^{\frac{1}{2}}\norm[auto]{h^{-1}\left(\bv_h-\cI_h^\bu \bv\right)}_H+\norm[auto]{\bv_h-\cI_h^\bu \bv}_V
        \right) ,
        \label{eq:20b}
    \end{align}
\end{subequations}
for all $\bv\in [H^r(\Omega)]^d$, $r_\bu\myeq\min\{r,k+1\}$, and $(q_h,\bv)\in Q_h\times V$. 
Many methods fall under the umbrella of this type of symmetric pressure
stabilization, and a comprehensive list and respective description can be 
found in \cite{Burman2009} and \cite{GarciaArchilla2020}.
The assumptions above ensure that the following modified inf-sup condition holds,
see \cite{Burman2009}[Lemma 3.1]. 
\begin{lemma}
    There exist two constants $C,\beta$, independent of $h$ and $\nu$, such that
\begin{equation}
    \beta\norm[auto]{q_h}_Q\leqslant \sup_{\bv_h\in [V_h]^d}\frac{\abs{b\inner[auto]{q_h}{\bv_h}}}{\norm[auto]{\bv_h}_V}
    +C\abs[auto]{q_h}_{j_h} \qquad \forall q_h\in Q_h .
    \label{eq:lemma inf-sup}
\end{equation}
\end{lemma}

Now we are in the position to present the semi-discrete formulation of the transient Stokes
problem~\eqref{eq:2} which is to
find $\bu_h\in H^1(J;V_h)$ and $p_h\in L^2(J;Q_h)$ such that
\begin{subequations}
	\label{eq:semidiscrete system}
    \begin{align}
        \inner[auto]{\partial_t\bu_h(t)}{v_h}+a\inner[auto]{\bu_h(t)}{\bv_h}+b\inner[auto]{p_h(t)}{\bv_h}&=\inner[auto]{\bf(t)}{\bv_h} , \\
        b\inner[auto]{q_h}{\bu_h(t)}-j_h\inner[auto]{q_h}{p_h(t)}&=0
    \end{align}
\end{subequations}
holds true in $L^2(J)$ for all $\bv_h\in V_h$ and all $q_h\in Q_h$. 
Under the above assumptions, the well-posedness of Problem \eqref{eq:semidiscrete system} 
is ensured, see for instance \cite[p.243]{Ern2021a}.
\begin{remark}
    If the spaces $V_h$ and $Q_h$ constitute an inf-sup stable pair, it is enough to take $j_h(\cdot,\cdot)\myeq 0$ to guarantee the well-posedness of the semi-discrete problem~\eqref{eq:semidiscrete system}.
\end{remark}

As a crucial tool in the forthcoming analysis,
we recall from \cite[equation~(3.16)]{Burman2009} the definition of the (Stokes) Ritz projection operator 
$ S_h \colon V\times Q\to V_h\times Q_h$:
for each $(\bu,p)\in V\times Q$, the velocity and pressure component of the Ritz projection $\left(S_h^\bu,S_h^p\right) := \cS_h(\bu,p) \in V_h \times Q_h$ are defined as the unique solution of Stokes problem
\begin{subequations}
	\label{eq:Ritz projection - eq}
    \begin{alignat}{3}
    a\inner[auto]{S_h^\bu}{\bv_h}+b\inner[auto]{S_h^p}{\bv_h}&=a\inner[auto]{\bu}{\bv_h}+b\inner[auto]{p}{\bv_h} 
    && \quad \forall \bv_h \in V_h
    ,
    \\
    b\inner[auto]{q_h}{S_h^\bu}-j_h\inner[auto]{S_h^p}{q_h}&=0
    && \quad \forall q_h \in Q_h.
    \end{alignat}
\end{subequations}
Problem \eqref{eq:Ritz projection - eq} is well posed thanks to the inf-sup condition~\eqref{eq:lemma inf-sup}, and satisfies a priori
stability estimate of the form
\begin{equation}
    \norm[auto]{S_h^\bu}_V^2+\abs[auto]{S_h^p}^2_{j_h}\leqslant C\left( \norm[auto]{\bu}^2_V+\norm[auto]{p}^2_Q \right) .
    \label{eq:22}
\end{equation}
Moreover, the following error estimates for the Ritz map were shown in \cite{Burman2009}.
\begin{lemma}
    Let $(\bu,p)\in C^1([0,T],[H^r(\Omega)]^d\cap V_{0,\div}\times H^s(\Omega))$ with $r\geqslant 2$ and $s\geqslant 1$. The following error estimate for the projection $S_h$ holds with $\alpha \in \{0,1\}$ (recall that $r_\bu = \min\{r,k+1\}$ and $s_p = \min\{s,\bar{l},l+1\}$):
    \begin{subequations}
    	\label{eq:23}
        \begin{align}
            \norm{\partial_t^\alpha(\bu-S_h^\bu)}_V+\abs{\partial_t^\alpha S_h^p}_{j_h}&\leqslant C\Bigl( \nu^{\frac{1}{2}}h^{r_\bu-1}\norm{\partial_t^\alpha\bu}_{r_\bu}+\nu^{-\frac{1}{2}}h^{s_p}\norm{\partial_t^\alpha p}_{s_p} \Bigr) , \label{eq:23a}\\
            \norm{p-S_h^p}_Q&\leqslant C\Bigl( \nu^{\frac{1}{2}}h^{r_\bu-1}\norm{\bu}_{r_\bu}+\nu^{-\frac{1}{2}}h^{s_p}\norm{p}_{s_p} \Bigr) , \label{eq:23b}
        \end{align}
    \end{subequations}
    for all $t\in [0,T]$ and $C>0$ independent of $\nu$ and $h$. Moreover,  provided the domain $\Omega$ is sufficiently smooth and if $\tilde{l}\geqslant 1$, an improved error estimate is given in the weaker $H \equiv L^2$-norm, see \cite[31.4, 32.3]{Ern2021b}. This reads
    \begin{equation}
        \norm{\partial_t^\alpha(\bu-S_h^\bu)}_H \leqslant Ch\Bigl(\norm{\partial_t^\alpha(\bu-S_h^\bu)}_V + \abs{\partial_t^\alpha S_h^p}_{j_h} \Bigr) .
        \label{eq:24}
    \end{equation}
\end{lemma}

\subsection{Time discretization using BDF methods}
For the temporal discretization, we consider the $q$-step backward difference formulae (BDF methods).
Let $t_n=n\tau,\;n=0,\ldots,N$ be a uniform partition of the time interval $[0,T]$ with time step $\tau=T/N$, then the BDF approximation of the time derivative is given by
\begin{equation}
    \bard \bu_h^n=\frac{1}{\tau}\sum_{i=0}^q \delta_i \, \bu_h^{n-i}, \qquad n \geqslant q ,
    \label{eq:26}
\end{equation}
where the method coefficients $\delta_i$ are determined from the relation
\begin{equation}
    \delta(\zeta)
    =
    \sum_{i=0}^q\delta_i\zeta^i
    =
    \sum_{l=1}^q\frac{1}{l}(1-\zeta)^l.
    \label{eq:25}
\end{equation}
The BDF methods are known to be stable and of classical order $q$ for $1\leqslant q \leqslant 6$, see \cite{Hairer1993,Hairer1996}. 
We assume $\bf\in C^0(\bar{J};L^2(\Omega))$ and set $\bf(t_n) = \bf^n$. 
The fully discrete system reads: for $n\geqslant q$ and given values $\bu^{n-1},\ldots,\bu^{n-q}$, find $(\bu_h^n,\;p_h^n)\in V_h\times Q_h$ such that
\begin{subequations}
	\label{eq:27}
	\begin{alignat}{3}
	    \inner[auto]{\bard\bu_h^n}{\bv_h}+a\inner[auto]{\bu_h^n}{\bv_h}+b\inner[auto]{p_h^n}{\bv_h}&=\inner[auto]{\bf^n}{\bv_h} 
        && \quad \forall \bv_h \in V_h, \label{eq:27a} \\
	    b\inner[auto]{q_h}{\bu_h^n}-j\inner[auto]{q_h}{p_h^n}&=0 
        && \quad \forall q_h \in Q_h.
        \label{eq:27b}
	\end{alignat}
\end{subequations}
For the fully discrete solution defined by
the collection of time step solution $\{\bu_h^n\}_{n=q}^N$, 
we also use the shorthand notation $\bu_h^{\tau}$, i.e.,
$\bu_h^{\tau}(t_n) = \bu_h^n$ for $n=q,\ldots, N$.
The initial values $\bu_h^i$, $i=0,\dotsc,q-1$, are either determined by a one-step method of order $q$, or by a lower order method with a sufficiently small step size.

\subsection{$G$-stability of BDF schemes}
\label{subsec:dahlquist+nevanlinna}

We collect here some results on $G$-stability for BDF schemes that allow us to use the energy estimates to prove stability. For more details we refer to \cite{Dahlquist1976,Dahlquist1978}.

\begin{lemma}[{Dahlquist's $G$-stability \cite{Dahlquist1976}}]
\label{lemma:Dahlquist}
Let $\delta(\zeta)$ and $\mu(\zeta)$ be polynomials of degree at most 
$q$ that have no common divisor. Let $(\cdot,\cdot)$ be an Euclidean inner product on $\bbR^N$ with associated Euclidean norm $|\cdot|$. If 
\begin{equation}
    \textnormal{Re}\frac{\delta(\zeta)}{\mu(\zeta)}>0 
    \qquad 
    \forall\zeta\in\mathbb{C}, \, |\zeta| < 1,
\end{equation}
then there exists a symmetric positive-definite (s.p.d.) matrix $G=[g_{ij}]\in
\mathbb{R}^{q\times q}$ and real numbers $\gamma_0,\ldots,\gamma_q$ such that for all $\bv_0,\ldots,\bv_q\in \bbR^N$,
\begin{equation}
    \textnormal{Re}\inner[auto]{\sum_{i=0}^q\delta_i\bv^{q-i}}{\sum_{j=0}^q\mu_j\bv^{q-j}
    }
    =\sum_{i,j=1}^qg_{ij}\inner[auto]{\bv^i}{\bv^j}
    -\sum_{i,j=1} ^{q}g_{ij}\inner[auto]{\bv^{i-1}}{\bv^{j-1}}
    +\abs[auto]{\sum_{i=0}^q\gamma_i
    \bv^i}^2.
    \label{eq:28}
\end{equation}
\end{lemma}

In order to apply Lemma~\ref{lemma:Dahlquist} to an arbitrary semi-inner product on some infinite dimensional function space, we need the following extension of the above lemma.
\begin{lemma}
    Let $\delta(\zeta)$ and $\mu(\zeta)$ be the polynomials of Lemma \ref{lemma:Dahlquist}. Let $(\cdot,\cdot)$ be a semi-inner product on a Hilbert space $H$ with associated norm $|\cdot|$. Then there exists a symmetric s.p.d. matrix $G=[g_{ij}]\in
    \mathbb{R}^{q\times q}$ 
    such that for all 
    $\bv_0,\ldots,\bv_q\in H$,
    \begin{equation}
        \textnormal{Re}\inner[auto]{\sum_{i=0}^q\delta_i\bv^{q-i}}{\sum_{j=0}^q\mu_j\bv^{q-j}
        }\geqslant \sum_{i,j=1}^qg_{ij}\inner[auto]{\bv^i}{\bv^j}-\sum_{i,j=1}
        ^{q}g_{ij}\inner[auto]{\bv^{i-1}}{\bv^{j-1}} .
        \label{eq:29}
    \end{equation}
    \label{lemma:Dahlquist1}
\end{lemma}
\begin{proof}
    See Appendix~\ref{sec:semi_proof}.
\end{proof}

The application of $G$-stability to BDF schemes is ensured by the following lemma.
\begin{lemma}[Multiplier technique of Nevanlinna and Odeh \cite{Nevanlinna1981}]
\label{lemma:NevanlinnaOdeh}
For $q=1, \ldots,5$ there exist $0\leqslant\eta_q <1$ such that for $\delta(\zeta)=\sum_{l=1}^q(1/l)(1-\zeta)^l$
and $\mu(\zeta)=1-\eta_q\zeta$,
\begin{align}
    \textnormal{Re}\frac{\delta(\zeta)}{\mu(\zeta)}>0
    \qquad      
    \forall\zeta\in\mathbb{C}, \, |\zeta| < 1 .
    \label{eq:30}
\end{align}
The classical values of $\eta_q$ from \cite{Nevanlinna1981}  are found to be 
\begin{equation}
	\eta_1 = \eta_2 = 0, \,\, \eta_3= 0.0769, \, \,\eta_4= 0.2878, \,\, \text{ and } \,\, \eta_5 = 0.8097.
\label{eq:eta_def}
\end{equation}
\end{lemma}
If there is no risk of confusion, we drop the index $q$ and simply write $\eta$.
For more details on multiplier values $\eta$ (e.g.~for the optimal multipliers values), we refer to \cite{Akrivis2015} and the references therein.

The previous two lemmata motivate the definition of a $G$-norm
associated with a semi-inner product $(\cdot, \cdot)$ 
on a Hilbert space $H$:
Given a collection of vectors $\bV^n=[\bv^n,\ldots,\bv^{n-q+1}] \subset H$,
we introduce the notation
\begin{align}
    |\bV^n|^2_{G}=\sum_{i,j=1}^qg_{ij}(\bv^{n-i+1},\bv^{n-j+1}) ,
    \label{eq:31}
\end{align}
where $G=[g_{ij}]$ is the s.p.d.~matrix appearing in \eqref{eq:27}. The identity \eqref{eq:31} defines a semi-norm on $H^q$ satisfying
\begin{align}
    c_0\abs{\bv^n}^2\leqslant c_0\sum_{j=1}^q\abs{\bv^{n-j+1}}^2\leqslant 
    |\bV^n|^2_G\leqslant c_1\sum_{j=1}^q\abs{\bv^{n-j+1}}^2,
    \label{eq:32}
\end{align}
where $c_0$ and $c_1$ denotes the smallest and the largest eigenvalue of $G$, respectively, and $\abs{\cdot}$ is the semi-norm induced by the semi-inner product $\inner{\cdot}{\cdot}$.

\begin{remark}
It was shown in~\cite{Akrivis2021a} that for six-step BDF method, there is no 
Nevanlinna--Odeh multiplier of the form $\mu(\zeta) = 1 - \eta_6 \zeta$,
with $1-|\eta_6| > 0$.
Instead, multipliers satisfying a certain relaxed positivity condition were found, which in turn allowed the authors of \cite{Akrivis2021a} to derive stability estimate for BDF-6 using
energy-type techniques.
In Section~\ref{sec:stab-bdf6}, we will discuss how
our stability analysis of fully discrete transient Stokes problem can be adapted to the case $q=6$
using the results from \cite{Akrivis2021a}.
\end{remark}

\section{Stability}
\label{sec:stability}
The goal of this section is to demonstrate how the stability of the fully discrete scheme~\eqref{eq:27}
can be established by mimicking
the approach for the continuous case, cf.~Section~\ref{subsection:continuous stability},
and taking advantage of the results on $G$-stability presented in Section~\ref{subsec:dahlquist+nevanlinna}. 
To obtain optimal stability estimates for the pressure, we here require that 
the initial data provided for the first $q-1$ steps are discretely divergence-free
as specified in the assumptions of \autoref{th:acceleration} and \autoref{th:pressure}, while
the more general case is briefly discussed in Section~\ref{sec:small_timestep}.  
Our analysis pivots around \eqref{eq:27} in the equivalent form
\begin{align}
    \inner[auto]{\bard\bu_h^n}{\bv_h}+a\inner[auto]{\bu_h^n}{\bv_h}+b\inner[auto]{p_h^n}{\bv_h}-b\inner[auto]{q_h}{\bu_h^n}+j_h\inner[auto]{q_h}{p_h^n}=\inner[auto]{\bf^n}{\bv_h}.
    \label{eq:33}
\end{align}
To emphasize the similarity between the fully discrete
stability estimate and its continuous counterpart \eqref{eq:10}, we
introduce the time-discrete (semi-)norm
\begin{equation}
    \label{eq:discrete_l2_norm}
\norm{\bw^{\tau}}_{\ell^2(J_q;\cH)}^2=\tau\sum_{n=q}^N \norm{\bw^n}_\cH^2
\end{equation}
for any sequence $\bw^{\tau} = (\bw^n)_{n=q}^N \subset \cH$ in a Hilbert space $\cH$
related to either the velocity or the pressure, where we set $J_q = (t_q, T)$.

\begin{theorem}
    \label{th:velocity}
    Let $\{(\bu_h^n,p_h^n)\}_{n=q}^N$ be the solution of the fully discrete problem \eqref{eq:27} with initial values $\bu_h^i\in V_h$, $i=0,\dotsc,q-1$. Then the following stability estimate holds for $0 \leqslant N \tau \leqslant T$,
    \begin{equation}
        \begin{aligned}
            \norm{\bu_h^N}_H^2 +\norm{\bu_h^\tau}_{\ell^2(J_q;V)}^2 + \abs{p_h^\tau}_{\ell^2(J_q;j_h)}^2\leqslant \; C\Biggl[\sum_{i=0}^{q-1}\norm{\bu_h^i}_H^2+\frac{c_P^2}{\nu}\norm{\bf}_{\ell^2(J_q;H)}^2\Biggr],
        \end{aligned}
        \label{eq:velocity_stability}
    \end{equation}
    with a constant $C>0$ independent of $h$, $N$, $\tau$ and the final time $T$.
\end{theorem}
\begin{proof}
    The main idea is to mimic the test procedure used to establish the stability estimate~\eqref{eq:stability estimate - continuous}
    for the continuous problem.
    To replace the integration with respect to the time variable in the fully discrete setting,
     Dahlquist's $G$-stability, Lemma~\ref{lemma:Dahlquist1}, and the multiplier technique of Nevanlinna and Odeh, Lemma~\ref{lemma:NevanlinnaOdeh} are combined to arrive at an estimate for a telescopic expression for $\|\bu_h^n\|^2_{H}$.
    For other parabolic-type PDEs, similar arguments were used in \cite{Lubich2013,Akrivis2015a,MCF}, and \cite{Akrivis2021}.
    
    We start by testing equation~\eqref{eq:33} with $(\bv_h,q_h) = (\bu_h^n-\eta \bu_h^{n-1},p_h^n)$ where $\eta$ is the multiplier defined in Lemma \ref{lemma:NevanlinnaOdeh} and depends on the order of the chosen BDF scheme.
    This yields
    \begin{equation}
    \begin{aligned}
        \inner[auto]{\bard\bu_h^n}{\bu_h^n-\eta \bu_h^{n-1}}&+a\inner[auto]{\bu_h^n}{\bu_h^n-\eta \bu_h^{n-1}}+b\inner[auto]{p_h^n}{\bu_h^n-\eta \bu_h^{n-1}} \\
        &-b\inner[auto]{p_h^n}{\bu_h^n}+j_h\inner[auto]{p_h^n}{p_h^n}=\inner[auto]{\bf^n}{\bu_h^n-\eta \bu_h^{n-1}}.
        \label{eq:34}
    \end{aligned}
    \end{equation}
    Thanks to the bilinearity of $b\inner{\cdot}{\cdot}$, two terms cancel immediately. 
    Recalling~\eqref{eq:31} we now define
    \begin{equation}
        |\bU^n_h|^2_{G}=\sum_{i,j=1}^qg_{ij}(\bu_h^{n-i+1},\bu_h^{n-j+1})
        \label{eq:35}
    \end{equation}
    and directly apply the \MT Lemma~\ref{lemma:NevanlinnaOdeh} in the form \eqref{eq:29} to the first term in \eqref{eq:34}. From $j_h\inner{q_h}{q_h}=\abs{q_h}_{j_h}^2$ and the coercivity of $a\inner{v_h}{v_h}= \norm{v_h}_V^2$, we obtain
    \begin{equation}
    \label{eq:36}
        \begin{aligned}
        \frac{1}{\tau}  \Big( |\bU_h^n|^2_G - |\bU_h^{n-1}|^2_G \Big) + \norm{\bu_h^n}_V^2 +\abs{p_h^n}^2_{j_h}-&\eta a\inner[auto]{\bu_h^n}{\bu_h^{n-1}}-\eta b\inner[auto]{p_h^n}{\bu_h^{n-1}} \\
        &\leqslant \inner[auto]{\bf^n}{\bu_h^n-\eta \bu_h^{n-1}}.
        \end{aligned}
    \end{equation}
    Thanks to~\eqref{eq:27b}, we have 
    $
    b\inner[auto]{p_h^n}{\bu_h^{n-1}} = 
    j_h\inner[auto]{p_h^n}{p_h^{n-1}}
    $
    for $n\geqslant q+1$, allowing us to rewrite \eqref{eq:36} as
    \begin{equation}
    \label{eq:37}
    \begin{aligned}
        \frac{1}{\tau}  \Big( |\bU_h^n|^2_G - |\bU_h^{n-1}|^2_G \Big) + \norm{\bu_h^n}_V^2 +\abs{p_h^n}^2_{j_h}-&\eta a\inner[auto]{\bu_h^n}{\bu_h^{n-1}}-\eta j_h\inner[auto]{p_h^n}{p_h^{n-1}} \\
        &\leqslant\inner[auto]{\bf^n}{\bu_h^n-\eta \bu_h^{n-1}}.
    \end{aligned}
    \end{equation}
    Next, invoking the the continuity of $a\inner{\cdot}{\cdot}$ and $j_h\inner{\cdot}{\cdot}$
    and successively applying a Cauchy--Schwarz inequality, a modified Young's inequality of the form $AB\leqslant \frac{1}{2\epsilon}A^2+\frac{\epsilon}{2} B^2$, and the Poincaré inequality, we estimate that
    \begin{subequations}
    	\label{eq:38}
        \begin{align}
            -\eta a\inner[auto]{\bu_h^n}{\bu_h^{n-1}}&\geqslant -\eta\norm{\bu_h^n}_V\norm{\bu_h^{n-1}}_V\geqslant -\eta \frac{\norm{\bu_h^n}_V^2}{2}-\eta \frac{\norm{\bu_h^{n-1}}_V^2}{2}, \\
            -\eta j_h\inner[auto]{p_h^n}{p_h^{n-1}}&\geqslant -\eta\abs{p_h^n}_{j_h}\abs{p_h^{n-1}}_{j_h}\geqslant -\eta \frac{\abs{p_h^n}_{j_h}^2}{2}-\eta \frac{\abs{p_h^{n-1}}_{j_h}^2}{2}, \\
            \inner[auto]{\bf^n}{\bu_h^n-\eta \bu_h^{n-1}}&= \inner[auto]{\bf^n}{\bu_h^n}-\eta  \inner[auto]{\bf^n}{\bu_h^{n-1}}  \\
            &\leqslant c_P\nu^{-\frac{1}{2}}\norm{\bf^n}_H\norm{\bu_h^n}_V+\eta c_P\nu^{-\frac{1}{2}}\norm{\bf^n}_H\norm{\bu_h^{n-1}}_V \nonumber \\
            &\leqslant\frac{c_P^2}{\nu}\frac{\norm{\bf^n}_H^2}{2\epsilon}+\frac{\epsilon}{2}\norm{\bu_h^n}_V^2+\eta\frac{c_P^2}{\nu}\frac{\norm{\bf^n}_H^2}{2\epsilon}+\eta\frac{\epsilon}{2}\norm{\bu_h^{n-1}}_V^2. \nonumber
        \end{align}
    \end{subequations}
    Combing these bounds, we arrive at
    \begin{equation}
        \begin{aligned}
            \frac{1}{\tau}  \Big( |\bU_h^n|^2_G - |\bU_h^{n-1}|^2_G \Big) +\frac{2-\eta-\epsilon}{2} \norm{\bu_h^n}_V^2 - \eta\frac{1+\epsilon}{2}\norm{\bu_h^{n-1}}_V^2\\ +\frac{2-\eta}{2}\abs{p_h^n}^2_{j_h} - \frac{\eta}{2}\abs{p_h^{n-1}}^2_{j_h}
            \leqslant (1+\eta)\frac{c_P^2}{2\epsilon\nu}\norm{\bf^n}_H^2 ,
        \end{aligned}
    \label{eq:39}
    \end{equation}
where $\epsilon>0$ is the free parameter of the modified Young's inequality. In order to be able to perform telescoping for the velocity in the $\norm{\cdot}_V$ norm, we must ensure that $\frac{2-\eta-\epsilon}{2} - \eta\frac{1+\epsilon}{2}>0$ and in order to do the same for the pressure, $\frac{2-\eta}{2}-\frac{\eta}{2}=1-\eta>0$ must hold. The latter condition $1-\eta>0$ is already satisfied thanks to the properties of the multipliers and regarding the first one, it is enough to choose $\epsilon$ such that $0<\epsilon<\frac{2-2\eta}{1+\eta}$. Multiplying~\eqref{eq:39} by $\tau$ and summing for $n = q+1, \ldots, N$ yields
    \begin{equation}
        \begin{aligned}
            |\bU_h^N|^2_G +\tau \frac{2-2\eta-\epsilon(1+\eta)}{2}\sum_{n=q+1}^N \norm{\bu_h^n}_V^2 + \tau(1-\eta)\sum_{n=q+1}^N \abs{p_h^n}^2_{j_h} \leqslant \;|\bU_h^{q}|^2_G \\
            +\tau\eta\frac{1+\epsilon}{2} \norm{\bu_h^q}_V^2+\tau\frac{\eta}{2}\abs{p_h^q}^2_{j_h}+\tau (1+\eta)\frac{c_P^2}{2\epsilon\nu}\sum_{n=q+1}^N \norm{\bf^n}_H^2,
        \end{aligned}
    \label{eq:40}
    \end{equation}
    where the sum is limited  to $n \geqslant q+1$ due to the validity range of the change from \eqref{eq:36} to \eqref{eq:37}.
    Exploiting the bounds \eqref{eq:32} for the \MT norm, we see that
    \begin{equation}
        \begin{aligned}
            \norm{\bu_h^N}_H^2 +\tau\sum_{n=q+1}^N \norm{\bu_h^n}_V^2 + \tau\sum_{n=q+1}^N \abs{p_h^n}^2_{j_h}\leqslant \; C\Biggl[\norm{\bu_h^q}_H^2 +\tau\eta\norm{\bu_h^q}_V^2+\tau\eta\abs{p_h^q}^2_{j_h}\\+\sum_{i=1}^{q-1}\norm{\bu_h^i}_H^2+\tau \frac{c_P^2}{\nu}\sum_{n=q+1}^N \norm{\bf^n}_H^2\Biggr].
        \end{aligned}
    \label{eq:41}
    \end{equation}
    To arrive at the final estimate, quantities for $n=q$ still have to be added to the left-hand side of 
    \eqref{eq:41}, and then need to be estimated in terms of the initial data. We proceed by testing \eqref{eq:34} at $n=q$ with $\bv_h=\bu_h^q$ and $q_h=p_h^q$, obtaining
    \begin{equation}
        \frac{\delta_0}{\tau}\norm{\bu_h^q}_H^2+\norm{\bu_h^q}_V+\abs{p_h^q}_{j_h} = \inner{\bf^q}{\bu_h^q}-\inner[auto]{\sum_{i=1}^q\frac{\delta_i}{\tau}\bu_h^{q-i}}{\bu_h^q}.
        \label{eq:42}
    \end{equation}
    Multiplying the resulting identity by $\tau$ combined with multiple applications of the Cauchy--Schwarz inequality, Young's inequality, 
	and the triangle inequality for the right-hand side leaves us with the estimate
    \begin{equation}
    \label{eq:44}
        \begin{aligned}
            \norm{\bu_h^q}_H^2 +\tau \norm{\bu_h^q}_V^2 + \tau \abs{p_h^q}^2_{j_h}
            \leqslant C\Biggl[ \sum_{i=0}^{q-1}\norm{\bu_h^{i}}^2 + \tau\frac{c_P^2}{\nu}\norm{\bf^q}_H^2 \Biggr] .
        \end{aligned}
    \end{equation}
    Combining the estimates \eqref{eq:41} and \eqref{eq:44} yields the desired stability bound
    \begin{equation}
        \begin{aligned}
            \norm{\bu_h^N}_H^2 +\tau\sum_{n=q}^N\norm{\bu^n_h}_{V}^2 + \tau\sum_{n=q}^N\abs{p_h^n}_{j_h}^2\leqslant \; C\Biggl[\sum_{i=0}^{q-1}\norm{\bu_h^i}_H^2+\tau\frac{c_P^2}{\nu}\sum_{n=q}^N \norm{\bf^n}_H^2\Biggr].
        \end{aligned}
        \label{eq:45}
    \end{equation}
\end{proof}

Next, we wish to derive a stability bound for the fully discrete pressure in a form analogous to \eqref{eq:16}. 
The main idea is again to follow the approach outlined in the continuous case in Section~\ref{subsection:continuous stability} with the following key differences:
\begin{itemize}
    \item[--] The continuous inf-sup condition \eqref{eq:4b} needs to
    be substituted by its discrete analogue \eqref{eq:lemma inf-sup};
    \item[--] the substitution of the continuous derivative with its discrete analogue \eqref{eq:26}.
\end{itemize} 

As in the continuous case, we will see that a bound for the discrete velocity derivative is needed, which we established first, and only afterwards we conclude with the derivation of the pressure bound.
\begin{theorem}
    \label{th:acceleration}
    Let $\{(\bu_h^n,p_h^n)\}_{n=q}^N$ be the solution of the fully
    discrete problem \eqref{eq:27} with initial values $\bu_h^i\in
    V_h$, $i=0,\dotsc ,q-1$. Suppose that the initial values $\{\bu_h^i\}_{i=0}^{q-1}$ are
    discretely divergence-free in the sense that
    for each $\bu_h^i\in V_h$ there exists corresponding $p_h^i\in Q_h$ such
    that
    \begin{equation}
        b\inner[auto]{q_h}{\bu_h^i}-j_h\inner[auto]{p_h^i}{q_h}=0, \qquad \forall q_h\in Q_h.
        \label{eq:discretely_div_free}
    \end{equation}
    Then the following stability estimate holds for $0 \leqslant N \tau \leqslant T$,
    \begin{equation}
        \norm{\bard\bu_h^\tau}_{
        \ell^2(J_q;H)
        }^2
        +\norm{\bu_h^N}_V^2+\abs{p_h^N}_{j_h}^2 \leqslant C\Biggl[\sum_{i=0}^{q-1}(\norm{\bu_h^{i}}_V^2+\abs{p_h^{i}}_{j_h}^2)
        +\norm[auto]{\bf}_{\ell^2(J_q; H)}^2
        \Biggr],
       \label{eq:acceleration_stability}
    \end{equation}
    with a constant $C>0$ independent of $h$, $N$, $\tau$ and the final time $T$.
\end{theorem}
\begin{proof}
    As in the proof of \autoref{th:velocity}, the key idea is
	to adapt the test procedure used in the continuous case, cf. Section~\ref{subsection:continuous stability},
    to the fully discrete setting by utilizing 
    Dahlquist's $G$-stability from Lemma~\ref{lemma:Dahlquist1}, and the multiplier technique of Nevanlinna and Odeh from Lemma~\ref{lemma:NevanlinnaOdeh}.

    We start by testing \eqref{eq:33} with $\bv_h = \bard\bu_h^n$ and $q_h=0$, yielding
    \begin{equation}
        \norm{\bard\bu_h^n}_H^2+a\inner[auto]{\bu_h^n}{\bard\bu_h^n}+b\inner[auto]{p_h^n}{\bard\bu_h^n} =\inner[auto]{\bf^n}{\bard\bu_h^n}.
        \label{eq:46}
    \end{equation}
    Compared to the continuous case \eqref{eq:12}, the non-symmetric
    bilinear form $b\inner{p_h^n}{\bard u_h^n}$ is still present since
    in general $\bu_h^n \notin V_{0,\div}$. 
    To handle this, we exploit that 
    $b\inner{\bu_h^{n}}{q_h}=j\inner{p_h^{n}}{q_h}$ for $0\leqslant n \leqslant N$
    thanks to~\eqref{eq:27b} and the assumption~\eqref{eq:discretely_div_free}
    on the initial values. With this in mind,
    we test \eqref{eq:33} with
    $\bv_h=0$ and $q_h=p_h^n$
    at $t_{n-i}$ for $n\geqslant q+1$ and $i=0,\ldots,q$ to deduce that
    \begin{equation}
        \begin{aligned}
        b\inner[auto]{p_h^n}{\bard\bu_h^n}=\frac{1}{\tau}\sum_{i=0}^q \delta_ib\inner{p_h^n}{\bu_h^{n-i}}=\frac{1}{\tau}\delta_i\sum_{i=0}^qj_h\inner{p_h^n}{p_h^{n-i}} 
        =j_h\inner{p_h^n}{\bard p_h^n}.
        \end{aligned}
        \label{eq:47}
    \end{equation}
    Consequently, we see that for $n\geqslant q+1$,
    \begin{equation}
        \norm{\bard\bu_h^n}_H^2+a\inner[auto]{\bu_h^n}{\bard\bu_h^n}+j_h\inner{p_h^n}{\bard p_h^n} =\inner[auto]{\bf^n}{\bard\bu_h^n}.
        \label{eq:48}
    \end{equation}
    To bound the left-hand side of \eqref{eq:48} from below, we would like to make use of $G$-stability and the multiplier technique. 
    This can be achieved as follows. First, we test \eqref{eq:27a} at time step $n-1$ with $\bv_h = -\eta\bard\bu_h^n\in V_h$ to obtain
    \begin{equation}
    \label{eq:49}
        \begin{aligned}
            -\eta\left[\inner[auto]{\bard\bu_h^{n-1}}{\bard\bu_h^n}+a\inner[auto]{\bu_h^{n-1}}{\bard\bu_h^n}+b\inner[auto]{p_h^{n-1}}{\bard\bu_h^n}\right]=-\eta\inner[auto]{\bf^{n-1}}{\bard\bu_h^n}.
        \end{aligned}
    \end{equation}
    Next, we repeat the argument followed in the derivation of \eqref{eq:48} maintaining the same range of validity to deduce
    \begin{equation}
    \label{eq:50}
        b\inner[auto]{p_h^{n-1}}{\bard\bu_h^n}=j_h\inner{p_h^{n-1}}{\bard p_h^n},
    \end{equation}
    and therefore
    \begin{equation}
        \begin{aligned}
            -\eta\left[\inner[auto]{\bard\bu_h^{n-1}}{\bard\bu_h^n}+a\inner[auto]{\bu_h^{n-1}}{\bard\bu_h^n}+j_h\inner[auto]{p_h^{n-1}}{\bard p_h^n}\right]=-\eta\inner[auto]{\bf^{n-1}}{\bard\bu_h^n}.
        \end{aligned}
        \label{eq:51}
    \end{equation}
    Adding \eqref{eq:51} to \eqref{eq:48}, we get
    \begin{equation}
        \begin{aligned}
        \norm{\bard\bu_h^n}_H^2-\eta \inner[auto]{\bard\bu_h^{n-1}}{\bard\bu_h^n}+a\inner[auto]{\bu_h^n-\eta \bu_h^{n-1}}{\bard\bu_h^n}+j_h\inner{p_h^n-\eta p_h^{n-1}}{\bard p_h^n} \\
        =\inner[auto]{\bf^n-\eta\bf^{n-1}}{\bard\bu_h^n}.
        \end{aligned}
        \label{eq:52}
    \end{equation}
    Note that for the moment, we only consider \eqref{eq:52} for $n\geqslant q+1$ since $\bard\bu_h^{n-1}$ is not defined for $n=q$.
    Analogously to \eqref{eq:31}, we define $G$-norm equivalents for the energy norm and the semi-norm
    induced by the pressure stabilization,
    \begin{equation}
        |\bU^n_h|^2_{F}=\sum_{i,j=1}^qg_{ij}a\inner{\bu_h^{n-i+1}}{\bu_h^{n-j+1}}, \quad
        |P^n_h|^2_{M}=\sum_{i,j=1}^qg_{ij}j_h\inner{p_h^{n-i+1}}{p_h^{n-j+1}}.
        \label{eq:53}
    \end{equation}
    Thanks to Lemma \ref{lemma:Dahlquist1}, the $G$-norms exist since $a\inner{\cdot}{\cdot}$ and  $j_h\inner{\cdot}{\cdot}$ define respectively an inner product and a semi-inner product. We apply the \MT to the third and fourth term in \eqref{eq:52} which yields the bound
    \begin{equation}
    \begin{aligned}
        \norm{\bard\bu_h^n}_H^2-\eta \inner[auto]{\bard\bu_h^{n-1}}{\bard\bu_h^n}+\frac{1}{\tau}\Bigl(|\bU^n_h|^2_{F}-|\bU^{n-1}_h|^2_{F}\Bigr)+\frac{1}{\tau}\Bigl(|P^{n}_h|^2_{M} - |P^{n-1}_h|^2_{M}\Bigr) \\
        \leqslant \inner[auto]{\bf^n-\eta\bf^{n-1}}{\bard\bu_h^n}.
    \end{aligned}
    \label{eq:54}
    \end{equation}
    Applying the Cauchy--Schwarz inequality and a scaled Young's inequality to the second and the last term in \eqref{eq:54}, upon rearranging we see that
    \begin{equation}
        \begin{aligned}
            \frac{1}{2}\norm{\bard\bu_h^n}_H^2-\frac{\eta}{2}\norm{\bard\bu_h^{n-1}}_H^2+\frac{1}{\tau}\Bigl(|\bU^n_h|^2_{F}-|\bU^{n-1}_h|^2_{F}\Bigr)+\frac{1}{\tau}\Bigl(|P^{n}_h|^2_{M} - |P^{n-1}_h|^2_{M}\Bigr) \\
            \leqslant \frac{1}{2(1-\eta)}\norm[auto]{\bf^n-\eta\bf^{n-1}}_H^2.
        \end{aligned}
        \label{eq:55}
    \end{equation}
    We sum up for $q+1\leqslant n\leqslant N$ and obtain
    \begin{equation}
        \begin{aligned}
            \tau\frac{1-\eta}{2}\sum_{n=q+1}^{N}\norm{\bard\bu_h^n}_H^2+|\bU^N_h|^2_{F}+|P^{N}_h|^2_{M}
            \leqslant \tau\frac{\eta}{2}\norm{\bard\bu_h^{q}}_H^2 +|\bU^{q}_h|^2_{F} +|P^{q}_h|^2_{M}\\+C \tau  \sum_{n=q}^N\norm[auto]{\bf^n}_H^2.
        \end{aligned}
        \label{eq:56}
    \end{equation}
    Exploiting once more the norm equivalence \eqref{eq:32} available for $G$-norms, we have
    \begin{equation}
    \label{eq:57}
        \begin{aligned}
            &\ \tau\sum_{n=q+1}^{N} \norm{\bard\bu_h^n}_H^2+\norm{\bu_h^N}_V^2+\abs{p_h^N}_{j_h}^2  \\
            \leqslant &\ C\Biggl[\tau\eta\norm{\bard\bu_h^{q}}_H^2 +\sum_{i=1}^{q}\Big(\norm{\bu_h^i}_V^2+\abs{p_h^i}_{j_h}^2\Big)
            +\tau \sum_{n=q}^N\norm[auto]{\bf^n}_H^2\Biggr].
        \end{aligned}
    \end{equation}
    As in the final steps of the proof of Theorem~\ref{th:velocity},
    the stability estimate is still to depend only on initial values
    and $\bf^n$. Thus, we need to both complete the sum
    $\tau\sum_{n=q}^{N}\norm{\bard\bu_h^n}_H^2$ on the left-hand side
    and bound the term $\tau\eta\norm{\bard\bu_h^{q}}_H^2
    +\norm{\bu_h^q}_V^2+\abs{p_h^q}_{j_h}^2$ on the right-hand side.
    This can be achieved at once with the following arguments. For
    $n=q$ we test \eqref{eq:33} with $\bv_h = \bard\bu_h^{q}$ and
    $q_h=\frac{\delta_0}{\tau} p_h^{q}$ (with $\delta_0 > 0$ being the leading
    coefficient of the BDF method, see \eqref{eq:25}). Collecting the resulting
    terms and using assumption \eqref{eq:discretely_div_free}, we
    obtain
    \begin{equation}
    \label{eq:58}
        \begin{aligned}
        &\ \norm{\bard\bu_h^{q}}_H^2+\frac{\delta_0}{\tau}\norm{\bu_h^{q}}_V^2+\frac{\delta_0}{\tau}\abs[auto]{p^{q}_h}_{j_h}^2 \\
        = &\ \inner[auto]{\bf^{q}}{\bard\bu_h^{q}}-\frac{1}{\tau}\left[a\inner[Big]{\bu_h^{q}}{\sum_{i=1}^q\delta_i\bu_h^{q-i}}+b\inner[Big]{p_h^{q}}{\sum_{i=1}^q\delta_i\bu_h^{q-i}}\right]\\
        = &\  \inner[auto]{\bf^{q}}{\bard\bu_h^{q}}-\frac{1}{\tau}\left[a\inner[Big]{\bu_h^{q}}{\sum_{i=1}^q\delta_i\bu_h^{q-i}}+j_h\inner[Big]{p_h^{q}}{\sum_{i=1}^q\delta_i p_h^{q-i}}\right] .
        \end{aligned}
    \end{equation}
    It will be crucial now that it holds $\delta_0>0$. We multiply \eqref{eq:58} by $\tau$ and then estimate the result by invoking the continuity of $a\inner{\cdot}{\cdot}$ and $j\inner{\cdot}{\cdot}$ combined with scaled Young's inequalities to see that
    \begin{equation}
        \begin{aligned}
        \tau\norm{\bard\bu_h^{q}}_H^2+\delta_0\Bigl(\norm{\bu_h^{q}}_V^2+\abs{p^{q}_h}_{j_h}^2 \Bigr)
        \leqslant
        &\frac{\tau}{2}\norm{\bard\bu_h^{q}}_H^2
        +\frac{\tau}{2}\norm{\bf^{q}}_H^2
        \\
        +\frac{\delta_0}{2}\Bigl(\norm{\bu_h^{q}}_V^2
        +\abs{p_h^{q}}_{j_h}^2\Bigr)
        &+\frac{1}{2\delta_0}
        \Biggl[
        \norm[Big]{\sum_{i=1}^q\delta_i\bu_h^{q-i}}_V^2
        +\abs[Big]{\sum_{i=1}^q\delta_i p_h^{q-i}}_{j_h}^2
        \Biggr], 
        \end{aligned}
        \label{eq:59}
    \end{equation}
    and consequently, after a kick-back argument, we deduce that
    \begin{equation}
        \begin{aligned}
        \tau\norm{\bard\bu_h^{q}}_H^2 +\norm{\bu_h^{q}}_V^2+\abs{p_h^{q}}_{j_h}^2
        \leqslant C\Biggl[ \sum_{i=0}^{q-1}(\norm{\bu_h^{i}}_V^2 + \abs{p_h^{i}}_{j_h}^2)+
        \tau\norm{\bf^{q}}_H^2 \Biggr] .
        \end{aligned}
        \label{eq:60}
    \end{equation}
    Now we use \eqref{eq:60} to add the contribution $\tau\norm{\bard\bu_h^q}_H^2$ with its relative bound to \eqref{eq:57} and 
    then employ it again to
    bound the remaining terms  with $n=q$ on the right-hand side of~\eqref{eq:57}
    to arrive at the desired estimate
    \begin{equation}
        \begin{aligned}
            \tau\sum_{n=q}^{N}\norm{\bard\bu_h^n}_H^2+\norm{\bu_h^N}_V^2+\abs{p_h^N}_{j_h}^2 \leqslant C\Biggl[\sum_{i=0}^{q-1}(\norm{\bu_h^{i}}_V^2+\abs{p_h^{i}}_{j_h}^2)+\tau \sum_{n=q}^N\norm[auto]{\bf^n}_H^2\Biggr] .
        \end{aligned}
        \label{eq:61}
    \end{equation}
    Recalling the definition of the discrete $\ell^2$-norms finishes the proof.
\end{proof}

We are now ready to state and prove the stability estimate for the pressure. For this part we again refer to Section~\ref{subsection:continuous stability} for the underlying ideas.
\begin{theorem}
    \label{th:pressure}
    Let $\{(\bu_h^n,p_h^n)\}_{n=q}^N$ be the solution of the fully discrete problem \eqref{eq:27} with initial values $\bu_h^i\in V_h$, $i=0,\dotsc ,q-1$ and assume that the initial values satisfy \eqref{eq:discretely_div_free}.
    Then the following stability estimate holds for $0 \leqslant N \tau \leqslant T$,
    \begin{equation}
        \norm[auto]{p_h^\tau}_{\ell^2(J_q;Q)}^2 \leqslant C \Biggl[\sum_{i=0}^{q-1}\Bigl( \norm{\bu_h^i}_H^2 +\frac{c_P^2}{\nu}\norm{\bu_h^{i}}_V^2+\frac{c_P^2}{\nu}\abs{p_h^{i}}_{j_h}^2 \Bigr) +\frac{c_P^2}{\nu}\norm{\bf}_{\ell^2(J_q;H)}^2 \Biggr] ,
        \label{eq:pressure_stability}
    \end{equation}
    with a constant $C>0$ independent of $h$, $N$, $\tau$ and the final time $T$.
\end{theorem}
\begin{proof}
    As in Section \ref{subsection:continuous stability}, the starting point is the modified inf-sup stability 
    \eqref{eq:lemma inf-sup} from which, together with \eqref{eq:24} and the usual procedures already seen, we deduce
    \begin{equation}
    \begin{aligned}
        \beta\norm[auto]{p_h^n}_Q \leqslant &\  \sup_{\bv_h\in V_h}\frac{\abs{b\inner[auto]{p_h^n}{\bv_h}}}{\norm[auto]{\bv_h}_V}+C\abs[auto]{p_h^n}_{j_h}\\
        = &\  \sup_{\bv_h\in V_h}\frac{\abs{-\inner[auto]{\bard\bu_h^n}{\bv_h}-a\inner[auto]{\bu_h^n}{\bv_h}-\inner[auto]{\bf^n}{\bv_h}}}{\norm[auto]{\bv_h}_V}+C\abs[auto]{p_h^n}_{j_h}  
        \\
        \leqslant 
        &\  c_p\nu^{-\frac{1}{2}}\norm{\bard\bu_h^n}_H+\norm[auto]{\bu_h^n}_V+c_p\nu^{-\frac{1}{2}}\norm[auto]{\bf^n}_H+C\abs[auto]{p_h^n}_{j_h}.
    \end{aligned}
    \label{eq:64}
    \end{equation}
    Squaring~\eqref{eq:64}, multiplying it by $\tau$ and summing from $n=q$ to $N$ yields 
    \begin{equation}
        \norm[auto]{p_h^\tau}_{\ell^2(J;Q)}^2\leqslant C\left[\frac{c_P^2}{\nu}\norm{\bard\bu_h^\tau}_{\ell^2(J_q;H)}^2+\norm[auto]{\bu_h^\tau}_{\ell^2(J_q;V)}^2+\abs[auto]{p_h^\tau}_{\ell^2(J_q;j_h)}^2+\frac{c_P^2}{\nu}\norm[auto]{\bf}_{\ell^2(J_q;H)}^2\right] .
        \label{eq:66}
	\end{equation}
	
    Combining \eqref{eq:66} with \eqref{eq:velocity_stability} and \eqref{eq:acceleration_stability} yields the desired result.
\end{proof}

\section{Stability for the six-step BDF method}
\label{sec:stab-bdf6}
In this section we sketch how our stability results can be extended to
the six-step BDF method.  We use the energy techniques based on
the multipliers first derived in \cite{Akrivis2021a}, wherein a multiplier for
the six-step BDF method was first found and has the form, see \cite[Proposition~2.3]{Akrivis2021a}:
\begin{equation}
\label{eq:BDF-6 multiplier}
	\begin{aligned}
		&\ \eta(\zeta) = 1 - \eta_1 \zeta^1 - \dotsb - \eta_6 \zeta^6 , 
		\qquad \text{where} \\
		&\ \eta_1 = \frac{13}{9}, \quad \eta_2 = -\frac{25}{36}, \quad \eta_3 = \frac{1}{9}, \quad \eta_4 = \eta_5 = \eta_6 = 0 .
	\end{aligned}
\end{equation}
In this section, by a slight abuse of notation with respect to \ref{subsec:dahlquist+nevanlinna}, the values $\eta_i$ are now all related to the multiplier for the BDF-6 method. 
It is crucial to note that this is not a multiplier in the Nevanlinna--Odeh sense, but a relaxed version which only satisfies a relaxed positivity assumption 
\begin{align}
    \label{eq:relax-pos-assump}
    1 - \eta_1 \cos(x) - \eta_2 \cos(2x) - \ldots - \eta_6 \cos(6x) > 0,
\end{align}
see Remark~2.1 and Definition~2.2 in \cite{Akrivis2021a}.
Further, \cite[Section~2.2]{Akrivis2021a} shows that for BDF-6 
there are no multipliers which
satisfy a strict positivity condition
\begin{align}
    \label{eq:pos-assump}
    1 - |\eta_1| - |\eta_2| - \ldots - |\eta_6| > 0.
\end{align}
In Section~3.1 and 3.2 of \cite{Akrivis2021a},
energy estimates for the six-step BDF method were derived considering
an abstract parabolic PDE in a Hilbert space.
In this section we show how the techniques from~\cite{Akrivis2021a}
can be employed to extend
Theorem~\ref{th:velocity} and \ref{th:acceleration} to the BDF-6
case.  Since many steps are similar to the proofs of the corresponding two theorems in Section~\ref{sec:stability} herein,
we only highlight the main differences.

\begin{theorem}[Theorem~\ref{th:velocity} for BDF-6]
	\label{th:velocity_bdf6}
    Let $\{(\bu_h^n,p_h^n)\}_{n=q}^N$ be the solution of the fully discrete problem \eqref{eq:27} with initial values $\bu_h^i\in V_h$, $i=0,\dotsc,q-1$ with $q=6$. Then the following stability estimate holds for $0 \leqslant N \tau \leqslant T$,
    \begin{equation}
        \begin{aligned}
            \norm{\bu_h^N}_H^2 +\norm{\bu_h^\tau}_{\ell^2(J_q;V)}^2 + \abs{p_h^\tau}_{\ell^2(J_q;j_h)}^2\leqslant \; C\Biggl[\sum_{i=0}^{q-1}\norm{\bu_h^i}_H^2+\frac{c_P^2}{\nu}\norm{\bf}_{\ell^2(J_q;H)}^2\Biggr],
        \end{aligned}
        \label{eq:velocity_stability_bdf6}
    \end{equation}
    with a constant $C>0$ independent of $h$, $N$, $\tau$ and the final time $T$.
\end{theorem}
\begin{proof}
According to \eqref{eq:BDF-6 multiplier}, we test equation \eqref{eq:33} with 
\begin{align*}
	\bv_h = &\  \bu_h^n - \frac{13}{9}  \bu_h^{n-1} + \frac{25}{36}  \bu_h^{n-2} - \frac{1}{9}  \bu_h^{n-3} := \bu_h^n - \vec\eta^T \bU_h^{n-1} , \\
	q_h = &\ p_h^n ,
\end{align*}
where we used the notation
\begin{align}
\bU_h^{n-1} = [\bu_h^{n-1} ; \dotsc ; \bu_h^{n-6}] \quad \text{and}\quad \vec\eta := \vec\eta_6 = [\tfrac{13}{9 }; -\tfrac{25}{36} ; \tfrac{1}{9} ; 0 ; 0 ; 0] \in \bbR^6.
\end{align}

This yields
\begin{equation}
	\begin{aligned}
		\inner[auto]{\bard\bu_h^n}{\bu_h^n - \vec\eta^T \bU_h^{n-1}} & + a\inner[auto]{\bu_h^n}{\bu_h^n - \vec\eta^T \bU_h^{n-1}}+b\inner[auto]{p_h^n}{\bu_h^n - \vec\eta^T \bU_h^{n-1} }  \\
		&-b\inner[auto]{p_h^n}{\bu_h^n}+j_h\inner[auto]{p_h^n}{p_h^n}=\inner[auto]{\bf^n}{\bu_h^n - \vec\eta^T \bU_h^{n-1} }.
	\end{aligned}
\end{equation}

Similarly as in the proof of Theorem~\ref{th:velocity}, the mixed terms $\pm b\inner[auto]{p_h^n}{\bu_h^n}$
cancel again, and following \cite[Section~3.1]{Akrivis2021a}, we use that Lemma~\ref{lemma:Dahlquist1}, \eqref{eq:27b}
is valid for the multiplier~\eqref{eq:BDF-6 multiplier} to
obtain the BDF-6 analogue of \eqref{eq:37}: for $n \geqslant q + 1 = 7$, we have
\begin{equation}
\label{eq:BDF-6 - 37}
\begin{aligned}
	\frac{1}{\tau}  \Big( |\bU_h^n|^2_G - |\bU_h^{n-1}|^2_G \Big) + \norm{\bu_h^n}_V^2 +\abs{p_h^n}^2_{j_h}- &\ a\inner[auto]{\bu_h^n}{\vec\eta^T \bU_h^{n-1}}- j_h\inner[auto]{p_h^n}{\vec{\eta}^T P_h^{n-1}} \\
	&\ \leqslant \inner[auto]{\bf^n}{\bu_h^n - \vec\eta^T \bU_h^{n-1}}.
\end{aligned}
\end{equation}

At this point, since we are not dealing with a classical Nevanlinna--Odeh multiplier (recall that $1-|\eta_1|-|\eta_2|-|\eta_3|<0$), 
and since a bound for $a\inner[auto]{\bu_h^n}{\vec\eta^T \bU_h^{n-1}} - j_h\inner[auto]{p_h^n}{\vec{\eta}^T P_h^{n-1}}$
suitable for telescoping cannot be found, a new argument is needed.
The solution developed in \cite{Akrivis2021a} is to avoid bounding these terms for each time-step and instead summing up and dealing with the resulting terms. In particular, summing \eqref{eq:BDF-6 - 37} from $n=6$ to $n=N$ gives
\begin{equation}
    \label{eq:BDF-6 - 38}
    \begin{aligned}
        |\bU_h^N|^2_G+ \tau\sum_{n=6}^Na\inner[auto]{\bu_h^{n}}{\bu_h^n-\vec\eta^T \bU_h^{n-1}} + &\ \tau \sum_{n=6}^Nj_h\inner[auto]{p_h^{n}}{p_h^n-\vec\eta^T P_h^{n-1}} \\
        &\ \leqslant |\bU_h^{5}|^2_G + \tau \sum_{n=6}^N\inner[auto]{\bf^n}{\bu_h^n-\vec\eta^T \bU_h^{n-1}}.
    \end{aligned}
\end{equation}
In \cite[Section~3.1]{Akrivis2021a}, the authors were able to bound the sums expression on the left-hand side of \eqref{eq:BDF-6 - 38} 
from below by recasting the sums into a weighted double sum of inner products where
the weights stem from a Toepliz matrix which turns out to be non-negative thanks
to relaxed positivity assumption \eqref{eq:relax-pos-assump}.
For the velocity the resulting estimate reads
    \begin{equation}
        \begin{aligned}
            \sum_{n=6}^Na\inner[auto]{\bu_h^{n}}{\bu_h^n-\vec\eta^T \bU_h^{n-1}}\geqslant \frac{1}{32}\sum_{n=9}^N \norm{\bu_h^n}^2-a\inner{\bu_h^6}{\mu_1\bu_h^5+\mu_2\bu_h^4 + \mu_3 \bu_h^3}\\
         -a\inner{\bu_h^{7}}{\mu_2\bu_h^5 + \mu_3\bu_h^4}-a\inner{\bu_h^{8}}{\mu_3\bu_h^5}.
        \end{aligned}
    \end{equation}
An analogous expression holds for the pressure term. Estimating the components involving the forcing term or starting approximations can be done with classical techniques as the one seen in \eqref{eq:42}, and detailed in \cite[Section~3.1]{Akrivis2021a}. Properties of the $G$-norms (e.g.~\eqref{eq:32}) then yield the stability estimate for the six-step BDF method.
\end{proof}

\begin{theorem}[Theorem~\ref{th:acceleration} for BDF-6]
    \label{th:acceleration_bdf6}
    Under the same hypothesis of Theorem~\ref{th:acceleration}, the following stability estimate holds for the BDF-q method with $q=6$ for $0 \leqslant N \tau \leqslant T$,
    \begin{equation}
        \norm{\bard\bu_h^\tau}_{
        \ell^2(J_q;H)
        }^2
        +\norm{\bu_h^N}_V^2+\abs{p_h^N}_{j_h}^2 \leqslant C\Biggl[\sum_{i=0}^{q-1}(\norm{\bu_h^{i}}_V^2+\abs{p_h^{i}}_{j_h}^2)
        +\norm[auto]{\bf}_{\ell^2(J_q; H)}^2
        \Biggr],
       \label{eq:acceleration_stability_bdf6}
    \end{equation}
    with a constant $C>0$ independent of $h$, $N$, $\tau$ and the final time $T$.
\end{theorem}
\begin{proof}
    The procedure is once again analogous to the proof of Theorem~\ref{th:acceleration}. Similar to the calculations in \eqref{eq:48}--\eqref{eq:52} we have to sum back in time to be able to use the multiplier technique. Testing at every time step with $\bard\bu_h^{n}$ and $\bard p_h^{n}$ yields
    \begin{equation}
        \label{eq:bdf-6-est-II}
        \begin{aligned}
            \left[\inner[auto]{\bard\bu_h^{n}}{\bard\bu_h^n}+a\inner[auto]{\bu_h^{n}}{\bard\bu_h^n}+j_h\inner[auto]{p_h^{n}}{\bard p_h^n}\right]&=\inner[auto]{\bf^{n}}{\bard\bu_h^n} , \\
            -\eta_1\left[\inner[auto]{\bard\bu_h^{n-1}}{\bard\bu_h^n}+a\inner[auto]{\bu_h^{n-1}}{\bard\bu_h^n}+j_h\inner[auto]{p_h^{n-1}}{\bard p_h^n}\right]&=-\eta_1\inner[auto]{\bf^{n-1}}{\bard\bu_h^n} , \\
            -\eta_2\left[\inner[auto]{\bard\bu_h^{n-2}}{\bard\bu_h^n}+a\inner[auto]{\bu_h^{n-2}}{\bard\bu_h^n}+j_h\inner[auto]{p_h^{n-1}}{\bard p_h^n}\right]&=-\eta_2\inner[auto]{\bf^{n-2}}{\bard\bu_h^n} , \\
            -\eta_3\left[\inner[auto]{\bard\bu_h^{n-3}}{\bard\bu_h^n}+a\inner[auto]{\bu_h^{n-3}}{\bard\bu_h^n}+j_h\inner[auto]{p_h^{n-3}}{\bard p_h^n}\right]&=-\eta_3\inner[auto]{\bf^{n-3}}{\bard\bu_h^n} .
        \end{aligned}
    \end{equation}
    
    Summing up the identities in \eqref{eq:bdf-6-est-II},
    we see that
    \begin{equation}
        \begin{aligned}
            \inner[auto]{\bard\bu_h^n-\vec\eta^T \dot\bU_h^{n-1}}{\bard\bu_h^{n}}&+a\inner[auto]{\bu_h^{n}-\vec\eta^T \bU_h^{n-1}}{\bard\bu_h^n}\\
            &+j_h\inner[auto]{p_h^{n}-\vec\eta^T \bP_h^{n-1}}{\bard p_h^n}=\inner[auto]{\bf^{n}-\vec\eta^T \bF_h^{n-1}}{\bard\bu_h^n},
            \label{eq:BDF-6 - 55}
        \end{aligned}
    \end{equation}
    where we set $\dot\bU_h^{n-1} \coloneqq [\bard\bu_h^{n-1} ; \dotsc ; \bard\bu_h^{n-6}]$ and $\bF^{n-1} \coloneqq [\bf^{n-1} ; \dotsc ; \bf^{n-6}]$.
    Then $G$-stability for the generalized multiplier yields
    \begin{equation}
        \begin{aligned}
            a\inner[auto]{\bu_h^{n}-\vec\eta^T \bU_h^{n-1}}{\bard\bu_h^n}+j_h\inner[auto]{p_h^{n}-\vec\eta^T P_h^{n-1}}{\bard p_h^n} \\
            \geqslant\frac{1}{\tau}\Bigl( |\bU_h^{n}|^2_F - |\bU_h^{n-1}|^2_F \Bigr) + \frac{1}{\tau}\Bigl( |P_h^{n}|^2_M - |P_h^{n-1}|^2_M \Bigr).
        \end{aligned}
    \end{equation}
    Finally, the remaining terms can be handled in the same way as in
    the previous proof using the properties of the associated Toepliz
    matrix, leading to
    \begin{equation}
        \begin{aligned}
            \sum_{n=9}^N\inner[auto]{\bard\bu_h^n-\vec\eta^T \dot\bU_h^{n-1}}{\bard\bu_h^{n}}\geqslant\frac{1}{32}\sum_{n=9}^N \norm{\bard\bu_h^n}^2-\inner{\bard\bu_h^9}{\mu_1\bard\bu_h^8+\mu_2\bard\bu_h^7 + \mu_3 \bard\bu_h^6}\\
         -\inner{\bard\bu_h^{10}}{\mu_2\bard\bu_h^8 + \mu_3 \bard\bu_h^7}-\inner{\bard\bu_h^{11}}{\mu_3\bard\bu_h^8} .
        \end{aligned}
    \end{equation}
     Bounding the components involving forcing terms and starting approximations can be done as explained in \cite[equation~(3.30)--(3.32)]{Akrivis2021a}, which are natural extensions of what seen in \eqref{eq:58}. The final estimate then again follows from the properties of the $G$-norms.
\end{proof}

Since  the pressure stability in \eqref{eq:66} is independent of the given BDF order, Theorem~\ref{th:velocity_bdf6} and Theorem~\ref{th:acceleration_bdf6} immediately imply the corresponding stability estimate for the pressure in the BDF-$6$ case.

\section{Convergence}
\label{sec:convergence}
In this section, we derive a priori error estimates for the fully discrete formulation \eqref{eq:31} for BDF schemes of order $1$ to $6$.
Thanks to the linearity of the transient Stokes problem and the stability estimates provided in Section~\ref{sec:stability} and Section~\ref{sec:stab-bdf6},
we can follow the standard recipe, see for instance~\cite[Chapter~73]{Ern2021a}:
First, we split the total discretization error into an interpolation and a discrete error.
Then we derive an error equation for the latter which has the form of the fully discrete transient Stokes problem
but with a ``small'' right-hand side (the  ``defect'') which encodes interpolation and consistency errors in space and time.
Finally, the stability estimates allow us to deduce convergence rates for the 
velocity and pressure approximation error from the corresponding rates for the defect.
Throughout this Section we will assume our domain $\Omega$ satisfies the required smoothness properties.

To obtain optimal error estimates we use the Ritz projection $(S_h^\bu(t_n),\; S_h^p(t_n)) \coloneqq \cS_h(\bu(t_n),\; p(t_n)) \in V_h \times Q_h$
defined in \eqref{eq:Ritz projection - eq}.
We decompose the velocity and pressure error into an interpolation error and a discrete error,
\begin{subequations}
    \begin{align}
        \bu(t_n)-\bu_h^n = &\ \underbrace{\bu(t_n)-S_h^\bu(t_n)}_{=:\be_\pi^n} + \underbrace{S_h^\bu(t_n)-\bu_h^n}_{=:\be_h^n} , \\
        p(t_n)-p_h^n = &\ \underbrace{p(t_n)-S_h^p(t_n)}_{=:\delta_\pi^n} + \underbrace{S_h^p(t_n)-p_h^n}_{=:\delta_h^n} .
    \end{align}
    \label{eq:67}
\end{subequations}
Recall that the discrete solution satisfies the discrete weak formulation
\begin{equation}
    \inner[auto]{\bard\bu_h^n}{\bv_h}+a\inner[auto]{\bu_h^n}{\bv_h}+b\inner[auto]{p_h^n}{\bv_h}-b\inner[auto]{q_h}{\bu_h^n}+j_h\inner[auto]{q_h}{p_h^n}=\inner[auto]{\bf^n}{\bv_h}
    \label{eq:68}
\end{equation}
for $(\bv_h, q_h) \in V_h \times Q_h$, cf.~\eqref{eq:33}.
Inserting the projected velocity $S_h^\bu(t_n)$ and projected
pressure $S_h^p(t_n)$ in place of the discrete solution into
\eqref{eq:68} and using the properties of the defining equation
\eqref{eq:Ritz projection - eq}, we deduce that the Ritz projection 
satisfies the fully-discrete transient Stokes equation up to some defect.
\begin{equation}
    \begin{aligned}
        \inner[auto]{\bard S_h^\bu(t_n)}{\bv_h}+a\inner[auto]{S_h^\bu(t_n)}{\bv_h}+b\inner[auto]{S_h^p(t_n)}{\bv_h}-b\inner[auto]{q_h}{S_h^\bu(t_n)}\\
        +j_h\inner[auto]{q_h}{S_h^p(t_n)}=\inner[auto]{\bf^n}{\bv_h}+\inner[auto]{\bd^n}{\bv_h},
    \end{aligned}
    \label{eq:69}
\end{equation}
where the defect $\bd^n$ is given by
\begin{equation}
    \bd^n= \bard S_h^\bu(t_n) - \partial_t\bu(t_n) = \bard \bu(t_n)-\partial_t\bu(t_n) -\bard\be_\pi^n.
\end{equation}
Thanks to linearity of the problem,
subtracting \eqref{eq:68} from \eqref{eq:69} leads to
\begin{equation}
    \inner[auto]{\bard\be_h^n}{\bv_h}+a\inner[auto]{\be_h^n}{\bv_h}+b\inner[auto]{\delta_h^n}{\bv_h}-b\inner[auto]{q_h}{\be_h^n}+j_h\inner[auto]{q_h}{\delta_h^n}=\inner[auto]{\bd^n}{\bv_h}.
    \label{eq:70}
\end{equation}
We are now in a similar situation to \eqref{eq:33} and it is enough to
bound the right-hand side in an appropriate way and to use the
stability estimates derived in \autoref{th:velocity},
\autoref{th:acceleration}, and \autoref{th:pressure}, or their
corresponding counterparts for BDF-$6$ from Section~\ref{sec:stab-bdf6}.
To this end, the defect satisfies the bound
\begin{equation}
    \begin{aligned}
        \norm{\bd^n}_H^2
        \leqslant \dfrac{C}{\tau} \Bigl( \tau^{2q}\norm{\partial_t^{(q+1)}\bu}_{L^2((t_{n-q},t_n);H)}^2 + h^{2r_\bu}\norm{\partial_t\bu}_{L^2((t_{n-q},t_n);H^{r_\bu}(\Omega))}^2 \Bigr),
        \label{eq:71}
    \end{aligned}
\end{equation}
which follows from standard approximation results, and a step-by-step derivation is provided in Appendix~\ref{sec:rhs_conv_bound_velocity}
for the reader's convenience. 
Consequently,
\begin{equation}
    \begin{aligned}
        \norm{\bd^\tau}_{\ell^2(J_q;H)}\leqslant  C \Bigl( \tau^{2q}\norm{\partial_t^{(q+1)}\bu}_{L^2(J;H)}^2 + h^{2r_\bu}\norm{\partial_t\bu}_{L^2(J;H^{r_\bu}(\Omega))}^2 \Bigr),
        \label{eq:71a}
    \end{aligned}
\end{equation}
where we wish to recall that $r_\bu \myeq \min\{r, k+1\}$, where
$\bu\in H^r(\Omega)$, and $k$ is the polynomial order of $V_h=V_h^k$,
see \eqref{eq:17a}. Analogously, we have $s_p\myeq\min\{ s, \tilde{l},
l+1\}$, where $p\in H^s(\Omega)$, $l$ is the polynomial order of
$Q_h=Q_h^l$ (see \eqref{eq:17b}), and $\tilde{l}$ is the order of weak
consistency of the stabilization. As a result, we obtain the following
estimate for the $\ell^{\infty}(J; H)$ norm of the velocity error.
\begin{theorem}
    \label{th: velocity consistency}
    Assume that $\bu\in H^1(J;[H^r(\Omega)]^d)\cap H^{q+1}(J;[L^2(\Omega)]^d)$ and $p\in C^0(J;H^s(\Omega))$ with $r\geqslant2$ and $s\geqslant1$.
    Let the initial values $\bu_h^i\in V_h$, $i=0,\dotsc ,q-1$ satisfy the approximation property \eqref{eq:20a}. Then the fully discrete numerical solution \eqref{eq:27} using BDF method of order $1,\dotsc,6$ satisfies the velocity error estimates
    \begin{equation}
        \begin{aligned}
        \norm{\bu_h^N-\bu(t_N)}_H^2\leqslant &\ C\Biggl[h^{2r_\bu}\norm{\bu(t_N)}_{r_\bu}^2\\
        &\ \hphantom{C\Biggl[} +\sum_{i=0}^{q-1}\Bigl((1+\nu) h^{2r_\bu}\norm{\bu(t_i)}_{r_\bu}^2 + \nu^{-1}h^{2(s_p+1)}\norm{p(t_i)}_{s_p}^2 \Bigr) \\
        &\ \hphantom{C\Biggl[} + \frac{c_P^2}{\nu}\Bigl( \tau^{2q}\norm{\partial_t^{(q+1)}\bu}_{L^2(J;H)}^2 +h^{2r_\bu}\norm{\partial_t\bu}_{L^2(J;H^{r_\bu}(\Omega))}^2 \Bigr)\Biggr],
        \end{aligned}
        \label{eq:velocity_cons_interp}
    \end{equation}
    \begin{equation}
        \begin{aligned}
        \norm{\bu_h^\tau-\bu}_{\ell^2(J_q;V)}^2\leqslant &\ C\Biggl[\nu h^{2(r_\bu-1)}\norm{\bu}_{\ell^2(J_q;H^{r_\bu})}^2\\
        &\ \hphantom{C\Biggl[} +\sum_{i=0}^{q-1}\Bigl((1+\nu) h^{2r_\bu}\norm{\bu(t_i)}_{r_\bu}^2 + \nu^{-1}h^{2(s_p+1)}\norm{p(t_i)}_{s_p}^2 \Bigr) \\
        &\ \hphantom{C\Biggl[} + \frac{c_P^2}{\nu}\Bigl( \tau^{2q}\norm{\partial_t^{(q+1)}\bu}_{L^2(J;H)}^2 +h^{2r_\bu}\norm{\partial_t\bu}_{L^2(J;H^{r_\bu}(\Omega))}^2 \Bigr)\Biggr].
        \end{aligned}
        \label{eq:velocity__bochner_cons_interp}
    \end{equation}

\end{theorem}
\begin{proof}
    The stability bounds from \autoref{th:velocity} and \autoref{th:velocity_bdf6}
    in conjunction with error equation~\eqref{eq:70} allows us to deduce that
    the discrete error is bounded by
    \begin{equation}
        \norm{\be_h^N}_H^2 +\norm{\be_h^\tau}_{\ell^2(J_q;V)}^2 + \abs{\delta_h^\tau}_{\ell^2(J_q;j_h)}^2\leqslant \; C\Biggl[\sum_{i=0}^{q-1}\norm{\be_h^i}_H^2+\frac{c_P^2}{\nu}\norm{\bd^\tau}_{\ell^2(J_q;H)}^2\Biggr].
        \label{eq:73}
    \end{equation}
    The only thing to clarify now is the order of convergence of $\norm{\be_h^i}_H^2$ for $i=0,\dotsc,q-1$. 
    For a general interpolant satisfying assumption~\eqref{eq:20}, we have
    \begin{equation}
        \norm{\be_h^i}_H^2 = \norm{S_h^\bu(t_i)- \cI_h^\bu\bu(t_i)}_H^2\leqslant \norm{S_h^\bu(t_i)-\bu(t_i)}_H^2 + \norm{\bu(t_i) - \cI_h^\bu\bu(t_i)}_H^2 .
        \label{eq:76}
    \end{equation}
    Invoking \eqref{eq:23}, \eqref{eq:24} and \eqref{eq:20} now leads to
    \begin{equation}
        \norm{\be_h^i}_H^2 \leqslant C\bigl[ (1+\nu) h^{2r_\bu}\norm{\bu(t_i)}_{r_\bu}^2 + \nu^{-1}h^{2(s_p+1)}\norm{p(t_i)}_{s_p}^2 \bigr].
        \label{eq:77}
    \end{equation}
   Inserting the bound for the defect given in~\eqref{eq:71}, shows that the discrete velocity error is bounded by
    \begin{equation}
        \begin{aligned}
        \norm{\be_h^N}_H^2 + \norm{\be_h^{\tau}}_{\ell^2(J_q,V)}^2
        \leqslant &\ C\Biggl[\sum_{i=0}^{q-1}\Bigl((1+\nu) h^{2r_\bu}\norm{\bu(t_i)}_{r_\bu}^2 + \nu^{-1}h^{2(s_p+1)}\norm{p(t_i)}_{s_p}^2\Bigr) \\
        &\ \hphantom{C\Biggl[} + \frac{c_P^2}{\nu}\Bigl( \tau^{2q}\norm{\partial_t^{(q+1)}\bu}_{L^2(J;H)}^2 +h^{2r_\bu}\norm{\partial_t\bu}_{L^2(J;H^{r_\bu}(\Omega))}^2 \Bigr)\Biggr].
        \end{aligned}
        \label{eq:78}
    \end{equation}
    To conclude the proof of the first estimate~\eqref{eq:velocity_cons_interp}, 
    we simply combine a triangle inequality 
    $
        \norm{\bu_h^N-\bu(t_N)}_H\leqslant \norm{\be_h^N}_H + \norm{\be_\pi^N}_H
    $
   with the bound 
   $\norm{\be_\pi^N}_H\leqslant Ch^{r_\bu}\norm{\bu(t_N)}_{r_\bu}$, which
   follows directly from the interpolation estimate~\eqref{eq:20a}.
   The same estimate also implies that
      $\norm{\be^\tau_{\pi}}_{\ell^2(J_q;V)}^2\leqslant C\nu h^{2(r_\bu-1)}\norm{\bu}_{\ell^2(J_q;H^{r_\bu})}^2$,
    which in conjunction with the triangle inequality
    $
    \norm{\bu_h^\tau-\bu}_{\ell^2(J_q;V)}\leqslant \norm{\be_h^\tau}_{\ell^2(J_q;V)} +\norm{\be^\tau_{\pi}}_{\ell^2(J_q;V)} 
    $ leads to our second estimate \eqref{eq:velocity__bochner_cons_interp}.
\end{proof}

\begin{theorem}
    \label{th: pressure consistency}
    Assume that $\bu\in H^1(J;[H^r(\Omega)]^d)\cap
    H^{q+1}(J;[L^2(\Omega)]^d)$ and $p\in C^0(J;H^s(\Omega))$ with
    $r\geqslant2$ and $s\geqslant1$. Set initial values as Ritz
    projection $(\bu_h^i,\; p_h^i) = (R_h^\bu(t_i),\; R_h^p(t_i)) =
    \cR_h(\bu(t_i),\; 0)$ thus satisfying both
    \eqref{eq:discretely_div_free} and \eqref{eq:20a}. Then the fully
    discrete numerical solution \eqref{eq:27} using BDF method of
    order $1,\dotsc,6$ satisfies the pressure error estimate
    \begin{equation}
        \begin{aligned}
            \norm{p-p_h^\tau}_{\ell^2(J_q;Q)}^2\leqslant &\ C\Biggl[\nu^{-1}h^{2(l+1)}\norm{p}_{\ell^2(J_q;H^{l+1})}^2 \\
            &\ \hphantom{C\Biggl[} + \nu h^{2r_\bu} \big( 1 + c_P^2 \nu^{-1} h^{-2} \big) \sum_{i=0}^{q-1} \norm{\bu(t_i)}_{r_\bu}^2 \\
            &\ \hphantom{C\Biggl[} + \nu^{-1}h^{2s_p} \big( h^2 + c_P^2 \nu^{-1} \big) \sum_{i=0}^{q-1} \norm{p(t_i)}_{s_p}^2 \\
            &\ \hphantom{C\Biggl[} +\frac{c_P^2}{\nu}\Bigl(\tau^{2q}\norm{\partial_t^{(q+1)}\bu}_{L^2(J;H)}^2 +h^{2r_\bu}\norm{\partial_t\bu}_{L^2(J;H^{r_\bu}(\Omega))}^2 \Bigr)\Biggr].
        \end{aligned}
        \label{eq:pressure_cons_ritz}
    \end{equation}

\end{theorem}
\begin{proof}
    The triangle inequality leads to
    \begin{equation}
        \norm{p_h^\tau-p}_{\ell^2(J_q;Q)}^2 \leqslant \norm{\delta_h^\tau}_{\ell^2(J_q;Q)}^2 + \norm{\delta_\pi^\tau}_{\ell^2(J_q;Q)}^2.
        \label{eq:79}
    \end{equation}
    By \eqref{eq:19}, we already have that $\norm{\delta_\pi^\tau}_{\ell^2(J_q;Q)}^2\leqslant C\nu^{-1}h^{2(l+1)}\norm{p}_{\ell^2(J_q;H^{l+1})}^2$. Moreover, according to \autoref{th:pressure}, we have
    \begin{equation}
        \norm[auto]{\delta_h^\tau}_{\ell^2(J_q;Q)}^2 \leqslant C \Biggl[\sum_{i=0}^{q-1}\Bigl(\norm{\be_h^i}_H^2 +\frac{c_P^2}{\nu}\norm{\be_h^i}_V^2+\frac{c_P^2}{\nu}\abs{\delta_h^i}_{j_h}^2 \Bigr)+\frac{c_P^2}{\nu}\norm{\bd^\tau}_{\ell^2(J_q;H)}^2 \Biggr] .
        \label{eq:80}
    \end{equation}
    Thanks to \eqref{eq:24}, we have
    \begin{equation}
        \label{eq:80b}
        \norm{\be_h^i}_H^2\leqslant C \Bigl[ \nu h^{2r_\bu}\norm{\bu(t_i)}_{r_\bu}^2 + \nu^{-1}h^{2s_p+2}\norm{p(t_i)}_{s_p}^2 \Bigr] .
    \end{equation}
    Next, we use the properties of the Ritz projection \eqref{eq:23a} to deduce that
    \begin{equation}
        \begin{aligned}
            \norm{\be_h^i}_V^2 +\abs{\delta_h^i}_{j_h}^2  = &\ \norm{S_h^\bu(t_i)-R_h^\bu(t_i)}_V^2 + \abs{S_h^p(t_i) - R_h^p(t_i)}_{j_h}^2 \\
            \leqslant &\ \norm{S_h^\bu(t_i)-\bu(t_i)}_V^2 + \abs{S_h^p(t_i) }_{j_h}^2 \\
            &\ +\norm{R_h^\bu(t_i)-\bu(t_i)}_V^2 + \abs{R_h^p(t_i) }_{j_h}^2 \\
            \leqslant &\ C \Bigl[ \nu h^{2(r_\bu-1)}\norm{\bu(t_i)}_{r_\bu}^2 + \nu^{-1}h^{2s_p}\norm{p(t_i)}_{s_p}^2 \Bigr] .
        \end{aligned}
        \label{eq:80a}
    \end{equation}
    Inserting~\eqref{eq:80b}, \eqref{eq:80a} and \eqref{eq:71a}
    into~\eqref{eq:80} yields the desired estimate for the discrete pressure error,
        \begin{align}
        \nonumber
        \norm{\delta_h^\tau}_{\ell^2(J;Q)}^2\leqslant &\ C\Biggl[\sum_{i=0}^{q-1} \Biggl( \nu h^{2r_\bu}\norm{\bu(t_i)}_{r_\bu} + \nu^{-1}h^{2(s_p+1)}\norm{p(t_i)}_{s_p}\\
        &\ \hphantom{C\Biggl[} +\frac{c_P^2}{\nu}\Bigl( \nu h^{2(r_\bu-1)}\norm{\bu(t_i)}_{r_\bu}^2 + \nu^{-1}h^{2s_p}\norm{p(t_i)}_{s_p}^2\Bigr)\Biggr)
        \label{eq:81}
        \\
        &\ \hphantom{C\Biggl[} +\frac{c_P^2}{\nu}\Bigl(\tau^{2q}\norm{\partial_t^{(q+1)}\bu}_{L^2(J;H)}^2 +h^{2r_\bu}\norm{\partial_t\bu}_{L^2(J;H^{r_\bu}(\Omega))}^2 \Bigr)\Biggr] . 
        \nonumber
        \end{align}
\end{proof}

\section{General initial data and small time-step limit}
\label{sec:small_timestep}
In this section, we briefly discuss
how the observations and results from
\cite{Burman2009}
regarding the use of general initial data for the velocity
can be generalized to our setting.
As already discussed in \cite{Burman2009} for BDF-$1$, we note that
the bounds on the pressure and its error derived in Section
\ref{sec:stability} and Section \ref{sec:convergence} rely strongly on the discrete divergence-free
assumption \eqref{eq:discretely_div_free} for the initial values and on the properties of the
Ritz projection \eqref{eq:Ritz projection - eq}. These conditions
allow us to untangle the dependency of the first $q$ computed values
of $p_h$ from the dependency on the first $q$ initial velocity values given
as input. This is clear considering that we can perform telescoping
only until \eqref{eq:50} holds. 

If the initial values do not satisfy
\eqref{eq:discretely_div_free}, the range of validity of \eqref{eq:55}
is restricted to $2q\leqslant n \leqslant N$. For the steps
$q\leqslant n \leqslant 2q-1$ we can argue similarly to what is done in
\eqref{eq:58} but we have to give up on the substitution of the
$b\inner{\cdot}{\cdot}$ form with $j_h\inner{\cdot}{\cdot}$
stabilization, leaving some bilinear forms
$b\inner{\cdot}{\cdot}$ explicit. Instead of
\eqref{eq:acceleration_stability}, we end up with the estimate
\begin{equation}
    \begin{aligned}
        \norm{\bard\bu_h^\tau}_{\ell^2(J_q;H)}^2
        + \norm{\bu_h^N}_V^2+\abs{p_h^N}_{j_h}^2 
        &\leqslant 
        C\Biggl[\sum_{i=0}^{q-1}\norm{\bu_h^{i}}_V^2
        +\norm[auto]{\bf}_{\ell^2(J_q;H)}^2 
        \\
        & \qquad +\sum_{i,j=0}^{q-1}b\inner{p_h^{q+j}}{\bu_h^{i}}\Biggr].
    \end{aligned}
    \label{eq:82}
\end{equation}
The same behaviour is then passed to the pressure stability, which becomes
\begin{equation}
    \norm[auto]{p_h^{\tau}}_{\ell^2(J_q;Q)}^2 
    \leqslant C 
    \Biggl[
        \sum_{i=0}^{q-1}\Bigl( \norm{\bu_h^i}_H^2 
        +\frac{c_P^2}{\nu}\norm{\bu_h^{i}}_V^2 \Bigr) 
        +\frac{c_P^2}{\nu}\norm{\bf}_{\ell^2(J_q;H)}^2
        +\frac{c_P^2}{\nu}\sum_{i,j=0}^{q-1}b\inner{p_h^{q+j}}{\bu_h^{i}} 
        \Biggr].
    \label{eq:83}
\end{equation}

This highlights the close interplay between the initial data and the
first $q$ computed pressures. If the initial $\bu_h^i$ are not weakly
divergence-free in the sense of assumption~\eqref{eq:discretely_div_free},
then a conditional pressure stability has to be expected.
This is what happens for example in \cite{Burman2009}, where the only
information on $\bu_h^i$ is that it is a general interpolant of the
exact initial velocity $\bu(t_i)$, which is by definition strongly
divergence-free. To guarantee pressure stability, \cite{Burman2009}
bounds the remaining $b(\cdot, \cdot)$ in the following way,
\begin{equation}
    \begin{aligned}
    b\inner{p_h^{q+j}}{\bu_h^{i}}= &\ b\inner{p_h^{q+j}}{\bu_h^{i}-\bu(t_i)} \leqslant \norm{p_h^{q+j}}_Q\norm{\bu_h^{i}-\bu(t_i)}_V\\
    \leqslant &\ \norm{p_h^{q+j}}_Q \nu^{\frac{1}{2}} h^{r_\bu-1}\norm{\bu(t_i)}_{r_\bu}\leqslant \frac{\nu h^{2(r_\bu-1)}}{2}\norm{p_h^{q+j}}_Q^2 +\frac{1}{2}\norm{\bu(t_i)}_{r_\bu}^2.
    \end{aligned}
    \label{eq:84}
\end{equation}

Under the inverse CFL-type assumption $\tau>Ch^{2(r_\bu-1)}$, one can absorb the
pressure-dependent terms in  \eqref{eq:84} on the left-hand side of
\eqref{eq:83}, thus restoring pressure stability. It should be noticed that
this comes to the expense of not only of a weak inverse CFL-type condition but
also of a higher regularity assumption for stability which would not have come into play
otherwise. We would also like to emphasize that the non-fulfilment of this weak inverse CFL-type condition
does not result into abnormal solution behavior involving 
non-physical oscillations or blow-ups.
Instead, $\delta$-scaled Young's inequality can be used to ensure that
$\tau-\delta Ch^{2(r_\bu-1)} > 0$ which in turn leads to 
correspondingly growing error constants,
causing an increasing initial error for time step decreasing below the critical inverse CFL value.

\section{Numerical experiments}
\label{section:numerics}
We conclude our presentation with a number of numerical experiments which corroborate our theoretical findings.

\subsection{Convergence studies}
To confirm the convergence rates predicted by \autoref{th: velocity consistency} and \autoref{th: pressure consistency} numerically, we conducted
a series of numerical experiments where
we used a manufactured solution and then computed the errors between the exact solution and the approximated one.
We considered problem \eqref{eq:1} in two dimensions with $\Omega = [0, 1] \times [0, 1]$, final time $T = 1$, and non-homogeneous boundary conditions. 
For easier comparison with the results presented in \cite{Burman2009}, the right-hand side $f$ and the boundary and initial data are chosen such that the exact solution is given by
\begin{subequations}
\begin{align}
        &u(x, y, t) = g(t)\left(
        \begin{aligned}
                &\sin(\pi x - 0.7) \sin(\pi y + 0.2) \\
                &\cos(\pi x - 0.7) \cos(\pi y + 0.2)
        \end{aligned}\right)
        , \\
        &p(x, y, t) = g(t)(\sin(x) \cos(y) + (\cos(1) - 1) \sin(1)) , 
\end{align}
\label{eq:manufactured-solution}
\end{subequations}
with $g(t) = 1 + 5t + e^{-10t} + \sin(t)$.
We focused on the case of inf-sup unstable equal-interpolation spaces where the non-zero stabilization $j_h\inner{\cdot}{\cdot}$
is given by the continuous
interior penalty stabilization (CIP, see \cite{Burman2006,Burman2007}).
 In all simulations below, 
 we used a $k$-step BDF time stepping scheme in time
given a finite element space with polynomials of degree $k$ in space.
 The initialization is always performed using the correspondent Ritz projection in order to guarantee optimal convergence. All computations have been performed using open source finite element software \href{https://ngsolve.org/}{NGSolve}, \cite{Schoeberl1997}, \cite{Schoeberl2014}.

\begin{figure}
    \centering
    \includegraphics[clip, trim = 1.5cm 14cm 0.cm 1.5cm, width=\textwidth]{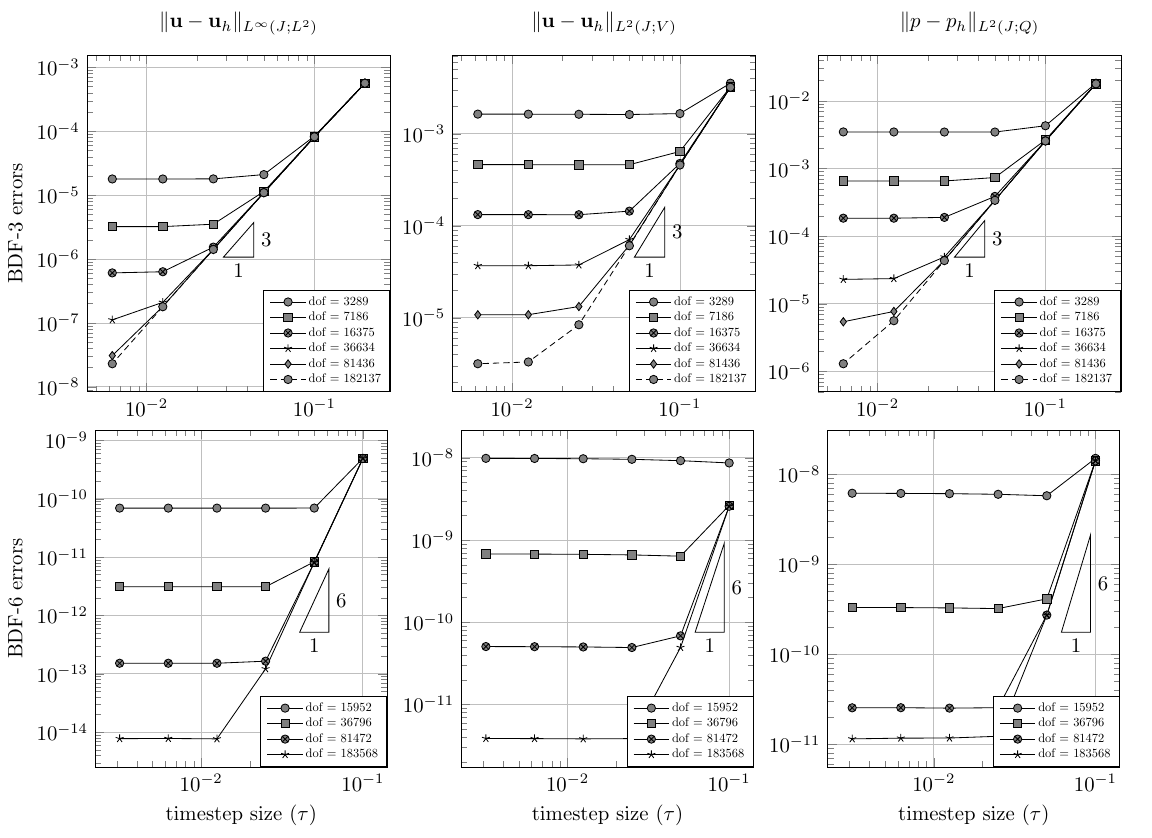}
    \caption{Temporal convergence studies for various mesh sizes. Top: BDF-3, the pair $\bbP^3$--$\bbP^3$ and CIP stabilization. Bottom: BDF-6, the pair $\bbP^6$--$\bbP^6$ and CIP stabilization.}
    \label{fig:BDF-CIP-Ritz}
\end{figure}

In \autoref{fig:BDF-CIP-Ritz} we plot the discrete analogous of the
norms $L^\infty(J;L^2)$, $L^2(J;H^1)$ and $L^2(J;L^2)$ (left to right) for the
quantities of interest versus the time step size for various spatial refinements. We test BDF-3 and BDF-6
with the pair $\bbP^3$--$\bbP^3$ and $\bbP^6$--$\bbP^6$ in space,
respectively. Note that
$\norm{\bv}_{L^2(J;V)}=\norm{\nu^{\frac{1}{2}}\nabla\bv}_{L^2(J;L^2)}$
and $\norm{\bv}_{L^2(J;Q)}=\norm{\nu^{-\frac{1}{2}}\bv}_{L^2(J;L^2)}$.
Both graphs show the predicted optimal-order convergence in time
as the error in space becomes smaller and smaller. 
The error curves flatten out since the spatial error starts dominating
for very small time steps.
For completeness, we also conducted convergence experiments in space which are
summarized in \autoref{fig:spaceBDF-CIP-Ritz} 
showing that the expected optimal spatial convergence order.

\begin{figure}
    \centering
    \includegraphics[clip, trim = 1.5cm 14cm 0.cm 1.5cm, width=\textwidth]{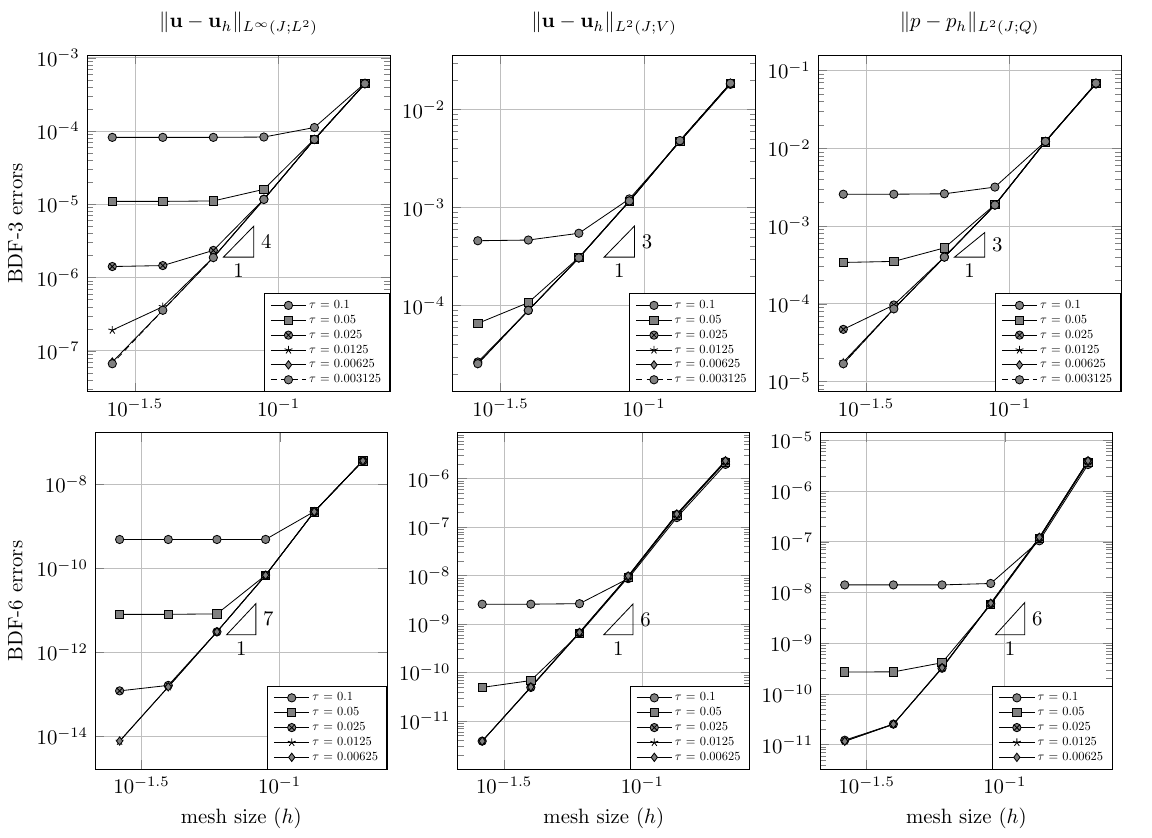}
    \caption{Spatial convergence studies for various time step lengths. Top: BDF-3, the pair $\bbP^3$--$\bbP^3$ and CIP stabilization. Bottom: BDF-6, the pair $\bbP^6$--$\bbP^6$ and CIP stabilization.}
    \label{fig:spaceBDF-CIP-Ritz}
\end{figure}

\subsection{Behavior in the small time step limit}
Finally, we briefly revisit the
observations and experiments in~\cite{Burman2009} 
and illustrate how the initial velocity approximation can affect
the error of the pressure approximation for small time steps also
in the case of higher-order BDF time stepping methods,
cf. Section~\ref{sec:small_timestep}.
For initial velocity approximations 
which are not discretely divergence free,
we expect a larger starting error
when the small time-step limit is not fulfilled.
As in \cite{Burman2009}, the exact solution
is chosen as in 
\eqref{eq:manufactured-solution} with $g(t) = 1$
and corresponding 
right-hand side $f$, boundary and initial data.
We compare the error in the pressure after the first
computed time step of the BDF scheme, i.e.
$$
\tau^{\frac{1}{2}}\norm{p(t_q)-p_h^q}_Q,
$$
with two different initial velocity approximation. In
\autoref{fig:small_timestep} the difference between the use of the
Ritz projection and the use of a Lagrange interpolant for BDF-3 and
BDF-6 is shown (right-hand side column). We verify that the
instability already seen in \cite{Burman2009} for BDF-1 extends to
higher order BDF schemes when the time-step decreases. Also, in the
same figure we show how the use of a Ritz projection for the initial
data resolves this issue in both cases (left-hand side column).
\begin{figure}[h]
    \centering
    \includegraphics[clip, trim = 1.5cm 12cm 1cm 1.5cm, width=0.9\textwidth]{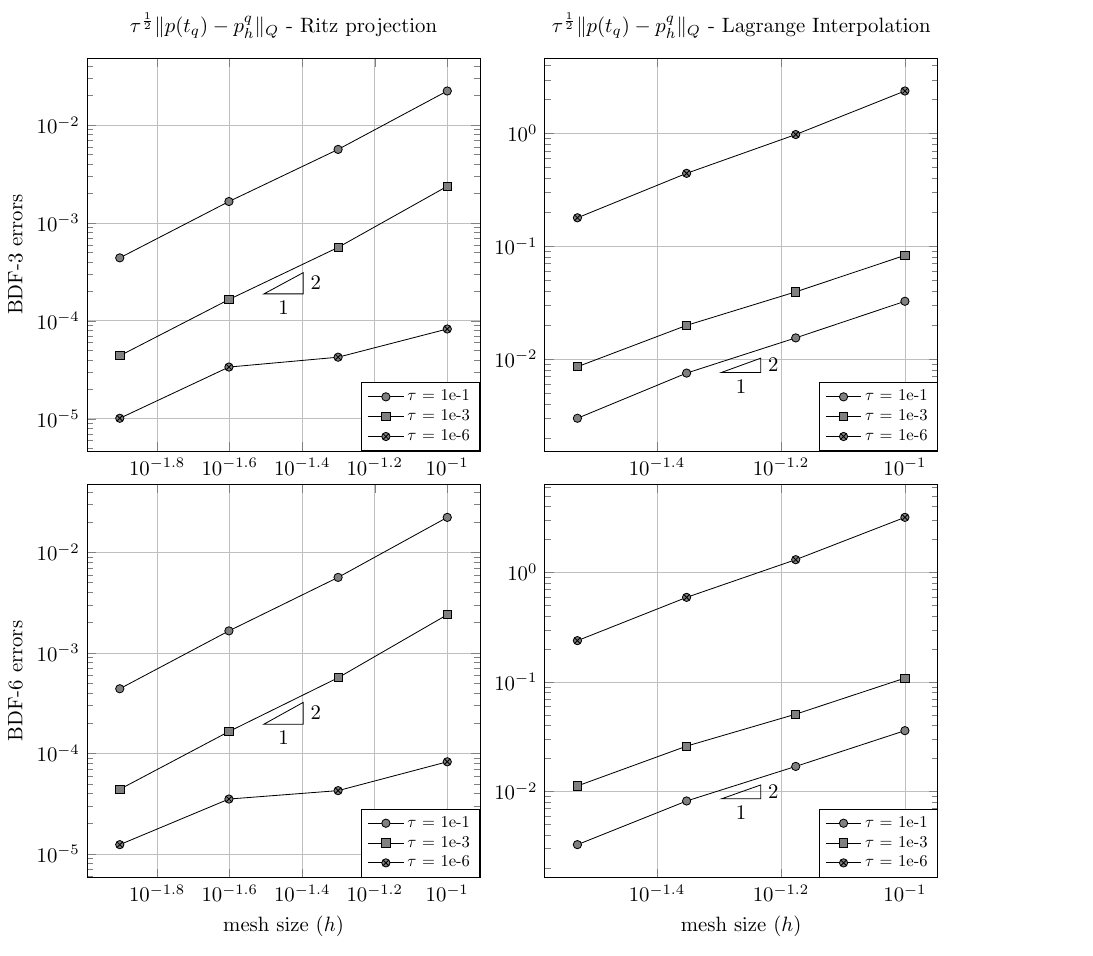}
    \caption{Pressure error studies in the small time-step limit using Ritz projection (left) and Lagrange interpolation (right) for the initial values. Top: BDF-3, the pair $\bbP^1$--$\bbP^1$ and CIP stabilization. Bottom: BDF-6, the pair $\bbP^1$--$\bbP^1$ and CIP stabilization.}
    \label{fig:small_timestep}
\end{figure}


\section*{Acknowledgements}
The work of Bal\'azs Kov\'acs is funded by the Heisenberg Programme of the Deut\-sche For\-schungs\-gemeinschaft (DFG, German Research Foundation) -- Project-ID 446431602. During the early preparation of the manuscript B.K.~was working at the University of Regensburg.

\appendix

\section{G-stability for semi-inner product}

\label{sec:semi_proof}
Lemma \ref{lemma:Dahlquist1} is equivalent, as stated in \cite[p.65]{Dahlquist1976}, to require that
\begin{equation}
    \sum_{j=0}^q\sum_{j=0}^q s_{ij}\inner{\bv^i}{\bv^j}\geqslant 0 ,
    \label{eq:appendix1}
\end{equation}
with
\begin{equation}
    S=[s_{ij}]=
    \begin{bmatrix}
        \delta_0 \\
        \vdots \\
        \delta_q
    \end{bmatrix}
    \begin{bmatrix}
        \mu_0 & \ldots & \mu_q
    \end{bmatrix}
    +
    \begin{bmatrix}
        \mu_0 \\
        \vdots \\
        \mu_q
    \end{bmatrix}
    \begin{bmatrix}
        \delta_0 & \ldots & \delta_q
    \end{bmatrix}
    +
    \begin{bmatrix}
        G & \mathbf{0}_q \\
        \mathbf{0}_q^T & 0
    \end{bmatrix}
    +
    \begin{bmatrix}
        0 & \mathbf{0}_q^T \\
        \mathbf{0}_q & G
    \end{bmatrix} , 
    \label{eq:appendix2}
\end{equation}
where $\mathbf{0}_q = [0,\ldots, 0]^T$ is the q-th zero vector. One can see that the matrix $S$ is independent from the semi-inner product chosen, but only dependent on the polynomials and thus positive semi-definite by hypothesis. Also, relation \eqref{eq:appendix1} holds for any semi-inner product by definition of semi-inner product, thus the Lemma is proven.

\section{Estimation of the defect}
\label{sec:rhs_conv_bound_velocity}

We want to bound the term $\norm{\bard \bu(t_n)-\partial_t\bu(t_n)-\bard\be_\pi^n}_H^2$. By the property of the BDF methods (see \cite{Hairer1993}) we have that
\begin{equation} 
    \bard\bu(t_n)-\partial_t\bu(t_n)=\bard R^n-R_t(t_n) ,
    \label{eq:92}
\end{equation}
where
\begin{equation}
    R(t)=\frac{1}{q!}\int_{t_{n-q}}^{t_n}(t-s)^q \partial_t^{(q+1)}\bu(s)\d s ,
    \label{eq:93}
\end{equation}
and the following estimate holds, where the Cauchy--Schwarz inequality is used in the last passage
\begin{equation}
    \begin{aligned}
        \norm{\bard\bu(t_n)-\partial_t\bu(t_n)}_H\leqslant C\Bigl( \tau^{q-1}\int_{t_{n-q}}^{t_n}
        \norm{\partial_t^{(q+1)}\bu(s)}_H\d s\Bigr)\\
        =C\left(\tau^{q-2}\int_{t_{n-q}}^{t_n}
        \tau\norm{\partial_t^{(q+1)}\bu(s)}_H\d s \right)
        \leqslant C\tau^{q-\frac{1}{2}}\norm{\partial_t^{(q+1)}\bu}_{L^2((t_{n-q},t_n);H)} .
        \label{eq:94}
    \end{aligned}
\end{equation}
Moreover, from  \eqref{eq:94}, \eqref{eq:77}, \eqref{eq:17} and using the fact that $\sum_{i=0}^q\alpha_i=0$ we have
\begin{equation}
    \begin{aligned}
        \norm{\bard\be_\pi^n}_H=\norm[auto]{\frac{1}{\tau}\sum_{i=0}^q\Bigl(\alpha_i\be_\pi^{n-i}-\alpha_i\be_\pi^{n-q}\Bigr) }_H = \norm[auto]{\frac{1}{\tau}\sum_{i=0}^q\alpha_i\int_{t_{n-q}}^{t_{n-1}}\partial_t\be_{\pi}(s)\d s }_H \\
        \leqslant C \tau^{-1}\sum_{i=0}^q \norm[auto]{\int_{t_{n-q}}^{t_{n-1}}\partial_t\be_{\pi}(s)\d s }_H \leqslant  C \tau^{-1}\sum_{i=0}^q \int_{t_{n-q}}^{t_{n-i}}\norm[auto]{\partial_t\be_{\pi}(s)}_H\d s \\
        \leqslant C \tau^{-\frac{1}{2}}\norm{\partial_t\be_{\pi}(s)}_{L^2((t_{n-q},t_n);H)} \leqslant C \tau^{-\frac{1}{2}}h^{r_\bu}\norm{\partial_t\bu}_{L^2((t_{n-q},t_n);H^{\br_u}(\Omega))} .
    \end{aligned}
    \label{eq:95}
\end{equation}
Overall
\begin{equation}
    \begin{aligned}
        \norm{\bard \bu(t_n)-\partial_t\bu(t_n)-\bard\be_\pi^n}_H^2\leqslant \norm{\bard \bu(t_n)-\partial_t\bu(t_n)}_H^2+\norm{\bard\be_\pi^n}_H^2\\
        \leqslant C \Bigl( \tau^{2q-1}\norm{\partial_t^{(q+1)}\bu}_{L^2((t_{n-q},t_n);H)}^2 +\tau^{-1}h^{2r_\bu}\norm{\partial_t\bu}_{L^2((t_{n-q},t_n);H^{r_\bu}(\Omega))}^2 \Bigr) .
    \end{aligned}
\end{equation}

\bibliographystyle{siamplain}
\bibliography{bibliography,bibliography_massing}

\end{document}